%% file: Diplom.tex
\documentclass[letterpaper,twoside,12pt]{report}
\pdfoutput=1

\usepackage[ngerman]{babel}
\usepackage[T1]{fontenc}
\usepackage{ae}
\usepackage{aecompl}
\usepackage[latin1]{inputenc}

\usepackage{amsmath}
\usepackage{amsthm}
\usepackage{amsfonts}
\usepackage{amssymb}
\usepackage[all]{xy}

\usepackage{yhmath}
\usepackage[scale=0.75, heightrounded]{geometry}
\usepackage{fancyhdr} 
\usepackage{mathrsfs}
\DeclareMathAlphabet{\mathpzc}{OT1}{pzc}{m}{it}

\begin{document} 
\allowdisplaybreaks
\sloppy 
\newenvironment{bew}{\begin{proof}[Beweis]}{\end{proof}}
\newtheorem{satz}{Satz}[chapter]
\newtheorem{bem}[satz]{Bemerkung}
\newtheorem{defi}[satz]{Definition}
\newtheorem{lem}[satz]{Lemma} 
\newtheorem*{beh}{Behauptung} 

\pagestyle{empty} 

\input{Deckblatt} 
\cleardoublepage

\cleardoublepage

\tableofcontents 
\cleardoublepage 

\pagestyle{fancy} 
\setlength{\headheight}{15.1pt}
\lhead[\thepage]{\bfseries\rightmark}
\rhead[\bfseries\leftmark]{\thepage}
\cfoot{\empty}

\setcounter{chapter}{-1}
\chapter{Einleitung}
\input{einleitung.tex}



\input{lim2.tex}
\chapter{Die Dachabbildung}
\input{dach.tex}

\section{Eigenschaften}
\input{orient.tex}

\section{Ausblick}
\input{ausblick.tex}

\addtocontents{toc}{\protect\vspace*{\baselineskip}}

\addcontentsline{toc}{chapter}{Literaturverzeichnis}
\bibliographystyle{hsiam}
\bibliography{bibliography} 

\end{document}

%% file: Deckblatt.tex
\begin{titlepage}

\begin{center}

\Large
\textsc{Die Dachabbildung in ganzzahliger \v{C}ech-Homologie}\\

\vspace{5cm}

\textsc{Diplomarbeit\\[0.5\baselineskip]
vorgelegt von\\[0.5\baselineskip]
Denise Nakibo\u glu\\
{\normalsize \textsc{geboren am 30. Mai 1982 in Coburg}}}\\

\vspace{4cm}
\textsc{Juli 2007}\\ 

\vspace{1cm}
\textsc{angefertigt im \\
Mathematischen Institut}\\
\textsc{der}

\textsc{Georg-August-Universität Göttingen}

\end{center}

\end{titlepage}

\newpage

%% file: einleitung.tex
Dem Wunsch entsprungen, das verallgemeinerte Borsuk-Ulam-Theorem aus \cite{borsuk} beziehungsweise \cite{Hage2006} für fundamentale Korrespondenzen durch Einbeziehung von Orientierung weiter zu verallgemeinern (denn die angegebenen Quellen beschäftigen sich vorrangig mit $\mathbb Z_2$-Koeffizienten), ist das Ziel dieser Arbeit, die Existenz der sogenannten Dachabbildung in \v{C}ech-Homologie mit $\mathbb  Z$-Koeffizienten nachzuweisen.\\ Die Existenz dieser Abbildung ist wesentlich für die Gültigkeit des verallgemeinerten Borsuk-Ulam-Theorems, denn sie hat die nützliche Eigenschaft, unter den passenden Voraussetzungen die Fundamentalklasse einer Mannigfaltigkeit auf "`die Fundamentalklasse"' der zugehörigen "`Giebelmannigfaltigkeit"' abzubilden, worauf im letzten Abschnitt der Arbeit näher eingegangen wird.\\
Es war also nötig, die Giebelkonstruktion $\widetriangle{(X,A)}$, die für ein topologisches Raumpaar $(X,A)$ durch\[
\widetriangle{(X,A)}:=(\widetriangle{X},\widetriangle{A}):=(p_\tau^X(X\times X),p_\tau^X(X\times A)\cup p_\tau^X(A\times X)\cup p_\tau^X(\bigtriangleup_X))\] definiert ist, auf ihre Orientierungseigenschaften zu untersuchen.\\ Hierbei bezeichnet  $p_\tau^X\colon X\times X\rightarrow X\times X/\tau$ die von der koordinatenvertauschenden Involution $\tau\colon X\times X\rightarrow X\times X$ induzierte Projektion.\\
Ebenso benötigt die Definition der Dachabbildung das simpliziale Kreuzprodukt, welches zu Beginn des zweiten Kapitels bezüglich seiner vorzeichenrelevanten Eigenschaften beleuchtet wird. \\
Im Entstehungsprozess der Arbeit stellte sich dann heraus, dass die gewünschte Verallgemeinerung der Dachabbildung nur in geraden Dimensionen möglich ist, genauer gesagt macht sogar die Definition der Dachabbildung nur in geraden Dimensionen Sinn. Die Ursache hierfür ist die Verstrickung der koordinatenvertauschenden Involution in die Definitionen, die in ungeraden Dimensionen "`die Orientierung umkehrt"', was im zweiten Kapitel nachgelesen werden kann.\\
Schließlich war es auch nötig, sich mit \v{C}ech-Homologie zu befassen, denn in dieser soll die Dachabbildung insbesondere existieren. Im ersten Kapitel werden daher die verschiedenen Möglichkeiten, die \v{C}ech-Homologie zu definieren, beschrieben und verglichen. Das Hauptresultat des ersten Kapitels wird für den entscheidenden Existenzbeweis für gerade Dimensionen benötigt.

%% file: lim2.tex
\chapter{Zwei Varianten der \v{C}ech-Homologie}

Eine Suche nach der Definition von \v{C}ech-Homologie wird möglicherweise zu zwei unterschiedlichen Ergebnissen führen. Das eine, das wohl als klassische Variante zu betrachten und zum Beispiel in \cite{MR0050886} zu finden ist, legt den inversen Limes der Homologiegruppen des Nervs der verschiedenen möglichen offenen Überdeckungen zu Grunde. Andererseits nutzt die beispielsweise in \cite{MR1335915} definierte Variante den inversen Limes der Homologiegruppen von offenen Umgebungen in einem umliegenden euklidischen Umgebungsretrakt. Dieses erste Kapitel soll die Isomorphie dieser beiden Definitionen zeigen.

\section{Limes}

Beide Definitionen der \v{C}ech-Homologie haben den inversen Limes gemein. Der sich anschließende erste Abschnitt des Kapitels soll Grundlegendes zum Thema "`inverser Limes"' zusammenfassen. Er basiert auf der Darstellung in \cite{MR1335915}.
Da ausschließlich vom inversen Limes die Rede sein wird, nennen wir diesen ab jetzt der Einfachheit halber schlicht "`Limes"'.

Einen kontravarianten Funktor $I\colon\Lambda\rightarrow \mathscr{K}$ zwischen Kategorien $\Lambda$ und $\mathscr{K}$ bezeichnen wir als Kofunktor. Eine (natürliche) Transformation $\varphi\colon I\rightarrow J$ zwischen zwei Kofunktoren $I,J\colon\Lambda\rightarrow\mathscr{K}$ ist eine Abbildung, die jedem Objekt $\lambda\in\Lambda$ einen Morphismus $\varphi_{\lambda}$ der Kategorie $\mathscr{K}$ zuordnet, nämlich $\varphi_{\lambda}\colon I\lambda\rightarrow J\lambda$, so dass für alle Morphismen $\alpha\colon\lambda\rightarrow\mu$ in $\Lambda$ folgendes Diagramm kommutiert: \[
\begin{xy}
  \xymatrix{
      I\mu \ar[r]^{\varphi_\mu} \ar[d]_{I\alpha}    &   J\mu \ar[d]^{J\alpha}  \\
      I\lambda \ar[r]_{\varphi_\lambda}             &   J\lambda   \\
  }
\end{xy}
\]
Bezeichnen wir für ein Objekt $K\in\mathscr{K}$ den durch dieses Objekt definierten konstanten Funktor $\Lambda\rightarrow\mathscr{K}$ ebenso mit $K$, so ist eine (natürliche) Transformation $\varphi\colon K\rightarrow I$ durch eine Familie von Morphismen $\left\{ \varphi_\lambda \colon K \rightarrow I \lambda \right\}_{\lambda\in\Lambda}$ gegeben, so dass für alle Morphismen $\alpha\colon\lambda\rightarrow\mu$ aus $\Lambda$ folgendes Diagramm kommutiert:\[
\begin{xy}
  \xymatrix{
                               &   I\mu \ar[dd]^{I\alpha}  \\
      K \ar[ru]^{\varphi_\mu} \ar[rd]_{\varphi_\lambda}  &  \\
                               &   I\alpha   \\
  }
\end{xy}
\] Das heißt, dass $\varphi_\lambda=I\alpha(\varphi_\mu)$ für alle Morphismen $\alpha\colon\lambda \rightarrow\mu$ ist.\\
Eine Transformation $u\colon L\rightarrow I$ mit $L\in\mathscr{K}$ heißt universell, falls es für alle Transformationen $\varphi\colon K\rightarrow I$ einen eindeutigen Morphismus $\psi\colon K\rightarrow L$ gibt, so dass $\varphi=u\psi$ ist, also falls\[
\mathscr{K}(K,L)\cong \text{Transf}(K,I), \qquad \psi\mapsto u\psi
\] gilt. 
\begin{defi}
Sei $I\colon\Lambda\rightarrow\mathscr{K}$ ein Kofunktor zwischen Kategorien $\Lambda$ und $\mathscr{K}$. Existiert eine universelle Transformation $u\colon L\rightarrow I$, so heißt $L$ der Limes von $I$, wofür die Schreibweise $L=lim(I)$ verwendet wird.
\end{defi}

Da der Limes bis auf Isomorphie eindeutig ist, ist seine Definition sinnvoll.

\begin{lem}[Eindeutigkeit des Limes bis auf Isomorphie]\label{eindlim}
Sei $I\colon\Lambda\rightarrow\mathscr{K}$ ein Kofunktor zwischen Kategorien $\Lambda$ und $\mathscr{K}$. Seien weiterhin $u\colon L\rightarrow I$ und $u^{\prime}\colon L^{\prime}\rightarrow I$ zwei universelle Transformationen. Dann gibt es einen eindeutigen Morphismus $\kappa\colon L\rightarrow L^{\prime}$, so dass $u=u^{\prime}\kappa$ und $\kappa\colon L\cong L^{\prime}$.
\end{lem}
\begin{bew}
Da $u$ universell ist, gibt es ein eindeutiges $\kappa^{\prime}\colon L^{\prime}\rightarrow L$ mit $u^{\prime}=u\kappa^{\prime}$. Aus der Universalität von $u^{\prime}$ bekommen wir genauso $\kappa\colon L\rightarrow L^{\prime}$ mit $u=u^{\prime}\kappa$. Es ist also\[
\begin{array}{lllll}
u  &=& u \kappa^{\prime}\kappa \qquad & \Rightarrow \kappa^{\prime}\kappa=id_L, &\text{ da u universell } \cr
u^{\prime} &=& u^{\prime}\kappa\kappa^{\prime} \qquad	& \Rightarrow \kappa\kappa^{\prime}=id_{L^{\prime}}, &\text{ da $u^{\prime}$ universell.}\cr
\end{array}\]
\end{bew}

In einigen Fällen gibt es Sätze für die Existenz eines Limes und sogar Aussagen darüber, wie dieser konkret aussieht. Später werden folgende Begriffsbildungen und der anschließende Satz von Bedeutung sein:

\begin{defi}[Quasiordnung, gerichtete Menge, inverses System]
Eine Relation $\lambda<\mu$ auf einer Menge $\Lambda$ heißt Quasiordnung, falls sie reflexiv und transitiv ist. Es ist möglich und oft sinnvoll, eine quasi-geordnete Menge $\Lambda$ als Kategorie aufzufassen, wobei die Objekte die Elemente von $\Lambda$ sind und jede Beziehung $,\lambda<\mu`$ ein Morphismus $\lambda\rightarrow\mu$ ist.\\
Eine solche Menge wird gerichtet genannt, falls es für alle Paare $\lambda,\lambda^{\prime}\in\Lambda$ ein $\mu\in\Lambda$ gibt, so dass $\lambda<\mu$ und $\lambda^{\prime}<\mu$ gilt. \\
Für eine quasi-geordnete Menge $\Lambda$ und eine Kategorie $\mathscr{K}$ definiert ein Kofunktor \linebreak[4]$I\colon\Lambda\rightarrow\mathscr{K}$ ein inverses System in $\mathscr{K}$. Für den Morphismus $\lambda<\mu$ in $\Lambda$ bezeichnet $I^\mu_\lambda$ den Morphismus $I^\mu_\lambda:=I(\lambda<\mu)\colon I\mu\rightarrow I\lambda$.   
\end{defi}

Der Limes eines inversen Systems existiert für die für uns wichtigen Kategorien immer und hat eine Darstellung, die später hilfreich im Umgang mit ihm sein wird.

\begin{satz}[Limes eines inversen Systems]\label{invlim}
Sei $I\colon\Lambda\rightarrow\mathscr{K}$ ein inverses System. Hier kann $\mathscr{K}$ beispielsweise die Kategorie der abelschen Gruppen, der Moduln oder der Komplexe sein. Dann existiert der Limes $lim(I)$ immer, und es ist\[
lim(I)=\left\{x=\left\{x_\nu\right\}\in\prod_{\nu\in\Lambda}{I\nu} \mid x_\lambda=I^\mu_\lambda x_\mu \quad \forall (\lambda<\mu)\right\}
\]
\end{satz}

\begin{bew}
Setze \[L:=\left\{x=\left\{x_\nu\right\}\in\prod\limits_{\nu\in\Lambda}I\nu \mid x_\lambda=I^\mu_\lambda x_\mu \quad \forall (\lambda<\mu)\right\}\] und sei $\varphi\colon K\rightarrow I$ eine Transformation für ein Objekt $K\in\mathscr{K}$.\\
Wegen der universellen Eigenschaft des Produktes gibt es einen eindeutigen Morphismus \[
\psi\colon K\rightarrow \prod\limits_{\nu\in\Lambda}{I\nu} \text{ , so dass } \psi(k)=\left\{\varphi_\lambda k\right\}_{\lambda\in\Lambda} \text{ für alle } k\in K.
\]
Für ein Element $\left\{\varphi_\lambda k\right\}_{\lambda\in\Lambda}$ aus dem Bild von $\psi$ gilt: $\varphi_\lambda(k)=I_\lambda^\mu\varphi_\mu(k)$ für alle $\lambda<\mu$, da $\varphi$ eine Transformation ist. Das heißt aber, dass $\psi$ eine Abbildung nach $L$ ist.\\
Sei $u\colon L\rightarrow I$ die Transformation, bei der $u_\lambda\colon L\rightarrow I\lambda $ die Projektion auf die $\lambda-$te Komponente ist. Dann gilt nach Konstruktion $\varphi=u\psi$ und $u$ ist somit universell.
\end{bew}

\section{Ein kofinaler Funktor}

Immer im Auge behaltend, dass wir die Isomorphie zweier Limes zeigen wollen, führen wir nun den Begriff des schwach und den des stark kofinalen Funktors ein. In Satz \ref{1.9} werden wir sehen, dass Kofinalität ein starkes Werkzeug ist, um Limes zu vergleichen.\\
Entscheidender Bestandteil des Abschnitts ist ein aufwändiges Beispiel eines stark kofinalen Funktors von der Kategorie der offenen Umgebungen eines Raumes $X$ in die zur Kategorie eines Polyeders unter $X$ duale Kategorie.

\subsection{Definitionen und Satz}

Die Begriffsbildung des kofinalen Funktors und der nachfolgenden Satz sind in \cite{MR1335915} nachzulesen, ebenso wie die Definition des konkreten kofinalen Funktors und die der zugehörigen Kategorien, die in diesem Abschnitt beleuchtet werden sollen. 
\bigskip

Schwache und starke Kofinalität wird wie folgt definiert:

\begin{defi}[Kofinalität] \label{kofi}
Ein Funktor $\Theta\colon \Omega\rightarrow\Lambda$ heißt schwach kofinal, falls es für alle $\lambda\in\Lambda$ einen Morphismus $\lambda\rightarrow\Theta(\omega)$ für ein $\omega\in\Omega$ gibt. Solch ein Funktor heißt stark kofinal, falls zusätzlich jedes Paar von Morphismen $\Theta\omega_1\leftarrow\lambda\rightarrow\Theta\omega_2$ mit $\omega_i\in \Omega$ und $\lambda\in\Lambda$  zu einem kommutativen Diagramm wie folgt vervollständigt werden kann:\[
\begin{xy}
				\xymatrix{
				                          & \Theta\omega_1 \ar[rd]^{\Theta g_1} &              \\
				 \lambda \ar[ru] \ar[rd]  &                                     & \Theta\omega \\
				                          & \Theta\omega_2 \ar[ru]^{\Theta g_2} &              
				}
\end{xy}
\]Hierbei sind $g_i\colon\omega_i\rightarrow \omega$ Morphismen in $\Omega.$
\end{defi}

Der nachfolgende Satz zeigt, wie nützlich die Begriffsbildung der Kofinalität für unsere Zwecke sein kann.

\begin{satz}\label{1.9}
Sei $\Theta\colon\Omega\rightarrow\Lambda$ ein Funktor. Für jeden Kofunktor $I\colon\Lambda\rightarrow\mathscr{K}$ und jedes Objekt $K\in\mathscr{K}$ definiere $E=I\Theta$ und\[
\hat{\Theta}\colon\text{Transf}(K,I)\rightarrow\text{Trans}(K,E) \text{\quad durch \quad } (\hat{\Theta}\varphi)=\varphi_{\Theta\omega}.
\]
\begin{enumerate}
	\item Ist $\Theta$ schwach kofinal, so ist $\hat{\Theta}$ injektiv.
	\item Ist $\Theta$ stark kofinal, dann ist $\hat{\Theta}$ bijektiv. Außerdem gilt: $u\colon L\rightarrow I$ ist universell $\Leftrightarrow                \hat{\Theta}(u)\colon L\rightarrow E$ ist universell.
				\[ \stackrel{Lemma \ref{eindlim}}{\Rightarrow} lim(E)\cong lim(I), \text{ falls einer dieser beiden existiert.}
				\]
\end{enumerate}
\end{satz}
\begin{bew}
\begin{itemize}
	\item[zu 1.)] Betrachte zuerst, warum für eine Transformation $\varphi\in\text{Transf}(K,I)$ das Bild $\psi:=\hat{\Theta}\varphi                          \in\text{Transf}(K,E)$ ist.	Sei hierfür $g\colon\omega_1\rightarrow\omega_2$ ein Morphismus in $\Omega$. Es ist\[
	             \begin{array}{llll}
	             \psi_{\omega_1} &=& \varphi_{\Theta\omega_1} & \cr
	             								 &=& I(\Theta g)(\varphi_{\Theta\omega_2}) & \text{ , da $\varphi$ eine Transformation ist}\cr
	             								 &=& Eg\psi_{\omega_2}. & \cr
	             								 \end{array}
							 \]Da dies zeigt, dass $\hat{\Theta}$ eine Abbildung wie gefordert ist, soll nun die Injektivität nachgewiesen werden:\\							 
							 Sei wie eben $\varphi\in\text{Transf}(K,I)$ eine Transformation. Da $\Theta$ schwach kofinal ist, existiert für alle $\lambda \in                   \Lambda$ ein Morphismus $f\colon\lambda\rightarrow\Theta\omega$ für ein $\omega\in\Omega$. Es gilt\[                                                \varphi_\lambda=If(\varphi_{\Theta\omega})=If(\hat{\Theta}\varphi)_\omega, \]da $\varphi$ eine Transformation ist. $\varphi$ wird                   auf diese Art und Weise durch $\hat{\Theta}\varphi$ ausgedrückt, woraus folgt, dass $\hat{\Theta}$ injektiv ist.
	\item[zu 2.)] Ein Inverses zu $\hat{\Theta}$ kann wie folgt konstruiert werden:\\
							 Für eine Transformation $\psi\in\text{Transf}(K,E)$ definiere $U\psi\in\text{Transf}(K,I)$ durch $(U\psi)_\lambda:=If\psi_\omega,$									 wobei $(\omega,f)$ passend zu $\lambda$ gewählt sind, so dass $f\colon\lambda\rightarrow\Theta\omega$ ein Morphismus in $\Lambda$ ist.		 
							  Im Allgemeinen wird die                    Definition von $U\psi$ von der Wahl von $(\omega,f)$ abhängen. Das ist nicht der Fall, wenn $\Theta$ stark kofinal ist. Denn für                  zwei Wahlen $f_1\colon\lambda\rightarrow\Theta\omega_1$ und $f_2\colon\lambda\rightarrow\Theta\omega_2$ können                  $g_i\colon\omega_i\rightarrow\omega$ für $i=1,2$ gefunden werden, so dass gilt: \begin{eqnarray}
							 \Theta g_1\circ f_1=\Theta g_2\circ f_2.
							 \label{stako}
							 \end{eqnarray}
							 Dann folgt\[
							 \begin{array}{llll}							
							 I(f_i)\psi_{\omega_i} &=& I(f_i)E(g_i)\psi_{\omega} & \text{ , da $\psi$ eine Transformation ist} \cr
							                       &=&I(f_i)I(\Theta(g_i))\psi_\omega & \text{ , nach der Definition von $E$} \cr
							                       &=&I(\Theta(g_i)f_i)\psi_\omega & \text{ , da $I$ ein Kofunktor ist.}							 
							 \end{array}
							 \]Der letzte Term hängt wegen (\ref{stako}) nicht von $i$ ab .\\
							 Es bleibt zu zeigen, dass $U\psi$ eine Transformation und $U$ invers zu $\hat{\Theta}$ ist. Für ersteres sei                                            $e\colon\mu\rightarrow\lambda$ ein Morphismus in $\Lambda$ und $f\colon\lambda\rightarrow\Theta\omega$ wie eben. Es ist\[
							 Ie(U\psi)_\lambda=IeIf\psi_\omega=I(fe)\psi_\omega=(U\psi)_\mu,
							 \] was genau die Transformationseigenschaft für $U\psi$ darstellt.
						   Weiter gilt\[
						   \begin{array}{lllllll}
							 (U\hat{\Theta}\varphi)_\lambda & \stackrel{Def.}{=} & If(\hat{\Theta}\varphi)_\omega &\stackrel{Def.}{=}&	If\varphi_{\Theta\omega}                 &\stackrel{\varphi\text{Trafo}}{=}& \varphi_\lambda \cr
							 (\hat{\Theta}U\psi)_\omega &\stackrel{Def.}{=}& (U\psi)_{\Theta\omega} &\stackrel{Def.}{=}& I(id)\psi_\omega &\stackrel{\psi                        \text{Trafo}}{=}& \psi_\omega,\cr                                                                                                                   \end{array}			 
							 \]also ist $U$ invers zu $\hat{\Theta}$ und $\hat{\Theta}$ somit bijektiv.\\
							 Für den zweiten Teil der Behauptung in 2. sei $u\colon L\rightarrow I$ universell.
							 Dann ist $\mathscr{K}(K,L)\cong\text{Transf}(K,I)$, und ebenso folgt aus der Universalität von $\hat{\Theta}(u)\colon L\rightarrow                   E$, dass $\mathscr{K}(K,L)\cong\text{Transf}(K,E)$ gilt.\\ Betrachte das folgende kommutative Diagramm:\[
							 \begin{xy}
							 		\xymatrix{ & \text{Transf}(K,I) \ar[dd]^{\cong} \\
							 		          \mathscr{K}(K,L)\ar[ru]^{,u\circ`} \ar[rd]_{,\hat{\Theta}u\circ`} & \\
							 		           & \text{Transf}(K,E)
							 		}
							 \end{xy} 
							 \]Ist der obere linke Pfeil ein Isomorphismus, so ist es auch der untere und umgekehrt. Das zeigt den zweiten Teil der Behauptung.
\end{itemize}
\end{bew}

Als nächstes behandeln wir das angekündigte Beispiel des kofinalen Funktors. Es ist zwar recht kompliziert, doch die Arbeit lohnt sich, denn wir werden es später benötigen. Der Funktor involviert die Kategorie der offenen Umgebungen eines topologischen Raumes $X$ sowie die Kategorie $\Lambda_{X}$, welche zur Kategorie eines Polyeders unter $X$ dual ist. Diese müssen wir nun zunächst definieren:

\begin{defi}[Kategorie der offenen Umgebungspaare]\label{umgkat}
Sei $(X,A)$ ein kompaktes topologisches Raumpaar in einem kompakten euklidischen Umgebungsretrakt $\mathrm E$. Wir definieren $\Omega_{(X,A)}$ als die Menge aller offenen Umgebungspaare von $(X,A)$ in $\mathrm E$, quasi-geordnet durch umgedrehte Inklusion (das heißt $(U,V)<(W,Z)$, falls $(W,Z)\subset (U,V)$ für $(U,V), (W,Z)\in\Omega_{(X,A)}$).\\
Sie ist sogar gerichtet, denn der Schnitt zweier offener Umgebungspaare ist wieder ein offenes Umgebungspaar, das in beiden ursprünglichen Umgebungspaaren enthalten ist, was genau die in der Definiton für eine gerichtete Menge geforderte Eigenschaft ist.\\ Die entsprechende Definition für offene Umgebungen eines kompakten topologischen Rau\-mes $X$ in einem umgebenden kompakten euklidischen Umgebungsretrakt bezeichnen wir mit $\Omega_X$.\\
Beide gerichteten Mengen fassen wir wie oben beschrieben als Kategorie auf.
\end{defi}

Die zweite Kategorie, die eine Rolle spielt, ist die Kategorie $\mathscr{P}^{X}$ eines "`Polyeders unter X"' beziehungsweise die Kategorie $\Lambda_{X}$ mit $\mathscr{P}^{X}=\Lambda_{X}^{op}$. Bevor wir aber ein "`Polyeder unter X"' verstehen können, sollten wir beginnen, ein Polyeder zu begreifen. Die nun folgenden simplizialen Grundlagen sind \cite{MR666554} entnommen und dienen dem besseren Verständnis eines Polyeders.

\begin{defi}[Geometrische Realisierung eines Simplizialkomplexes]
Es sei ein nichtleerer Simplizialkomplex $K$ gegeben. Diesem kann ein topologischer Raum $|K|$ zugeordnet werden, der geometrische Realisierung von $K$ genannt wird. Definiere hierfür $|K|$ als die Menge aller Funktionen $\alpha$ von der Eckenmenge von $K$ nach $I$, so dass gilt:
\begin{itemize}
	\item Für alle $\alpha$ ist $\left\{v \in K \mid \alpha(v)\neq 0\right\}$ ein Simplex von $K$. Insbesondere ist dann $\alpha(v)\neq 0$ nur für eine endliche Menge von Ecken.
	\item Für alle $\alpha$ ist $\sum\limits_{v\in K}{\alpha(v)}=1$.
\end{itemize}
Ist $K$ leer, so definiere $|K|=\emptyset$. Die reelle Zahl $\alpha(v)$ wird $v$-te baryzentrische Koordinate von $\alpha$ genannt. \\
Um die Topologie auf $|K|$ zu definieren, betrachte für ein Simplex $s$ in $K$ die Menge \[|s|:=\left\{\alpha \in |K| \mid \alpha(v)\neq 0 \Rightarrow v\in s \right\},\] die als abgeschlossenes Simplex bezeichnet wird. Der Raum $|K|$ erhält die schwache Topologie, das heißt $A\subset|K|$ ist abgeschlossen (oder offen) in $|K|$, genau dann, wenn $A\cap|s|$ abgeschlossen (oder offen) in $|s|$ für alle $s\in K$ ist.
\end{defi}

Die Zuordnung $|\cdot|$ ist sogar ein Funktor von der Kategorie der Simplizialkomplexe in die der topologischen Räume. Denn sei eine simpliziale Abbildung $f\colon K_1\rightarrow K_2$ zwischen zwei Simplizialkomplexen $K_1$ und $K_2$ gegeben, so wird durch \[
|f|(\alpha)(v^\prime)=\sum_{f(v)=v^\prime}\alpha(v) \text{ für } v^\prime \in K_2
\] eine stetige Abbildung $|f|\colon |K_1|\rightarrow |K_2|$ definiert.

\begin{bem}
Für jeden Simplizialkomplex $K$ ist seine geometrische Realisierung $|K|$ ein normaler Hausdorff-Raum.\footnote{Siehe beispielsweise \cite{MR666554}(Thm 3.1.17).}
\end{bem}

Nicht nur die Definition des abgeschlossenen Simplex ist wichtig, auch das sogenannte offene Simplex ist eine nützliche Definition im Umgang mit der Topologie von Simplizialkomplexen:

\begin{defi}[Offenes Simplex]
Sei $K$ ein Simplizialkomplex. Definiere das offene Simplex $<s>\subset |K|$ durch\[
<s>:=\left\{\alpha \in |K| \mid \alpha(v)\neq 0 \Leftrightarrow v\in s \right\}.
\]
Ein offenes Simplex ist nicht unbedingt offen in $|K|$ aber immer offen in $|s|$.\\
Die offenen Simplizes eines Komplexes partitionieren diesen im folgenden Sinne: Jeder Punkt $\alpha\in|K|$  gehört zu einem eindeutigen offenen Simplex von $K$. Dieses ist das Simplex $<s>$, wobei $s=\left\{v\in K\mid \alpha(v)\neq 0\right\}$ ist.
\end{defi}

Zusätzlich dazu, dass jeder Punkt in $|K|$ in einem eindeutigen offenen Simplex wiedergefunden werden kann, gibt es auch für jeden Punkt eine Darstellung als Konvexkombination von Elementen aus $|K|$.
Denn werden die Ecken $v$ eines Simplizialkomplexes $K$ mit ihrer charakteristischen Funktion  
\[
v(v^\prime)=\left\{\begin{array}{ll}
													0 & \text{ , falls  $v\neq v^\prime$}\\
													1 & \text{ , falls  $v=v^\prime$,}
													\end{array}
\right. \] in $|K|$ identifiziert, so kann jedes $\alpha\in |K|$ wie folgt als Konvexkombination von Punkten aus $|K|$ geschrieben werden: \[
\alpha=\sum\limits_{v\in K}{\alpha(v)v}. \]

\begin{bem}\label{statop}
Es ist möglich, eine weitere Topologie auf der geometrischen Realisierung eines Komplexes $K$ zu definieren, indem die gröbste Topologie gewählt wird, so dass die durch die eben definierten charakteristischen Funktionen induzierten Abbildungen $|K|\rightarrow [0,1]$ stetig sind. Diese Topologie wird starke Topologie genannt. Für lokal endliche Komplexe stimmt sie mit der schwachen überein. \footnote{Vgl. \cite{MR1335915}, Remark 7.14}
\end{bem}

Vorangegangenes hilft uns,  die Topologie der Polyeder besser zu verstehen, denn:

\begin{defi}[Polyeder]
Eine Triangulierung $(K,f)$ eines topologischen Raumes $X$ besteht aus einem Simplizialkomplex $K$ und einem Homöomorphismus $f\colon|K|\rightarrow X$. Hat $X$ eine Triangulierung, so wird $X$ ein Polyeder genannt.
\end{defi}

Es ist uns schließlich möglich, die Kategorie eines "`Polyeders unter $X$"' und den stark kofinalen Funktor zu definieren.

\begin{defi}[Polyeder unter X]\label{poly}
Sei $X$ ein topologischer Raum. Definiere die Kategorie $\mathscr{P}^{X}$ als Kategorie mit Homotopieklassen von Abbildungen $\left[\xi\colon X\rightarrow P_\xi\right]$ von $X$ in ein Polyeder $P_\xi$ als Objekte. Ein Morphismus zwischen $\xi$ und $\eta$ ist eine Homotopieklasse von Abbildungen $\left[\alpha\colon P_\xi\rightarrow P_\eta\right]$, so dass $\alpha\xi\simeq\eta$ ist. Die Komposition ist gegeben durch Komposition der Abbildungen $\alpha$.\\ Die Kategorie $\mathscr{P}^{X}$ wird ein Polyeder unter $X$ genannt.
\end{defi}

Die duale Kategorie $K^{op}$ zu einer Kategorie $K$ bezeichnet die Kategorie, die dieselben Objekten wie $K$ als Objekte hat, in der allerdings die Morphismen in die "`andere"' Richtung zeigen. Das heißt, dass es für jeden Morphismus $f\colon A\rightarrow B$ in $K$ einen Morphismus $f^{op}\colon B\rightarrow A$ in $K^{op}$ gibt (diese Zuordnung soll injektiv sein) und die Komposition $f^{op}g^{op}=(gf)^{op}$ in $K^{op}$ definiert ist, genau dann wenn $gf$ in $K$ definiert ist. 

\begin{satz} \label{2.4}
Sei $X$ eine kompakte Teilmenge eines kompakten euklidischen Umgebungsretraktes $\mathrm{E}$. 
Dann gibt es einen Funktor $\vartheta$ zwischen $\Omega_{X}$ und der Kategorie $\Lambda_{X}$, wobei $\Lambda_{X}^{op}=\mathscr{P}^{X}$ ist, definiert durch folgende Vorschrift:\[
\begin{array}{lrll}
\vartheta\colon & \Omega_{X}  &\rightarrow & \Lambda_{X}\cr 
&U & \mapsto &  \left[X\hookrightarrow U\right] \cr
&\left(U<V\right) & \mapsto & [V\hookrightarrow U],
\end{array}
\] wobei in den Homotopieklassenklammern jeweils die Inklusionsabbildungen stehen.\\
$\vartheta$ ist stark kofinal.
\end{satz}

Sowohl Polyeder als auch Homotopieklassen von Abbildungen kommen im vorherigen Satz vor. Mit diesen umgehen zu können ist also entscheidend. Um Satz \ref{2.4} beweisen zu können, ist daher etwas Vorarbeit nötig, die wir im folgenden Einschub leisten möchten.

\input{simp2.tex}

\section{Definitionen der \v{C}ech-Homologie}
Es gibt zwei verschiedene Arten, die \v{C}ech-Homologie zu definieren, welche für viele Räume isomorph zueinander sind. Diese sollen nun zunächst eingeführt werden. Die erste und klassische Variante basiert auf \cite{MR0050886}, die zweite auf \cite{MR1335915}. Sind Koeffizienten nicht explizit angegeben, so sind Homologiegruppen mit Koeffizienten in $\mathbb Z$ gemeint.

\subsection{\v{C}ech-Homologie (Variante 1)}
Für ein topologisches Raumpaar $(X,A)$ sei $Cov(X,A)$ die Menge aller offenen Überdeckungen des Paares, das heißt die Menge aller offenen Überdeckungen $\mathcal{U}$ von $(X,A)$, so dass es $\mathcal{V}\subset\mathcal{U}$ mit $A\subset\bigcup\limits_{V\in\mathcal{V}}V$ gibt. Wir schreiben für eine solche Überdeckung auch $(\mathcal{U},\mathcal{V})$.\\ Diese Menge ist quasi-geordnet durch die Relation der Verfeinerung von Überdeckungen:\\
Wir nennen eine Überdeckung $(\mathcal{U},\mathcal{V})$ feiner als eine Überdeckung $(\mathcal{W},\mathcal{Z})$, falls jede Menge aus $\mathcal{U}$ in einer Menge aus $\mathcal{W}$ enthalten ist und zudem jede Menge aus $\mathcal{V}$ in einer Menge aus $\mathcal{Z}$ enthalten ist. Wir notieren dann $(\mathcal{W},\mathcal{Z})<(\mathcal{U},\mathcal{V})$.

\begin{lem}\label{gercov}
$Cov(X,A)$ ist mit dieser Relation sogar gerichtet.
\end{lem}

\begin{bew}
Seien $(\mathcal{U}_1,\mathcal{V}_1)$ und $(\mathcal{U}_2,\mathcal{V}_2)$ aus $Cov(X,A)$ gegeben. Definiere $(\mathcal{W},\mathcal{Z})$ durch \linebreak[4] $\mathcal{W}=\left\{U_1\cap U_2 \mid \text{ $U_1 \in \mathcal{U}_1$ und $U_2\in\mathcal{U}_2$} \right\}$ und $\mathcal{Z}=\left\{V_1\cap V_2 \mid \text{ $V_1 \in \mathcal{V}_1$ und $V_2\in\mathcal{V}_2$ }\right\}$. Diese Überdeckung erfüllt $(\mathcal{U}_1,\mathcal{V}_1)<(\mathcal{W},\mathcal{Z})$ und $(\mathcal{U}_2,\mathcal{V}_2)<(\mathcal{W},\mathcal{Z})$.
\end{bew}

\begin{defi}[Nerv einer Überdeckung]
Wir definieren den Nerv $\nu(\mathcal{U},\mathcal{V})$ einer Überdeckung $(\mathcal{U},\mathcal{V})\in Cov(X,A)$ als das simpliziale Paar $(\nu\mathcal{U},\nu_A\mathcal{V})$.\\ Hierbei ist $\nu\mathcal{U}$ der Simplizialkomplex, dessen Ecken die nichtleeren Elemente $U\in\mathcal{U}$ sind, dessen $n-$Simplizes also $(n+1)-$Tupel $(U_0,U_1,\ldots,U_n)$ mit $U_i\in\mathcal{U}$ und $\bigcap\limits_{i}U_i\neq\emptyset$ sind. Dieser beinhaltet einen Unterkomplex  $\nu_A\mathcal{V}$ bestehend aus den Simplizes $s$ von $\nu\mathcal{U}$, für die \[{\mathit A}\; \cap \bigcap\limits_{V \text{ Ecke von } s}V\neq \emptyset \text{ gilt} .\]
Im Folgenden stehe $|\nu(\mathcal{U},\mathcal{V})|:=(|\nu\mathcal{U}|,|\nu_A\mathcal{V}|)$ für die geometrische Realisierung eines Nervs $\nu(\mathcal{U},\mathcal{V})$.
\end{defi}

Um den erwünschten Kofunktor zu definieren, dessen Limes die \v{C}ech- Homologie ist, fehlt jetzt noch eine Zutat:

\begin{defi}[Projektion]\label{proj}
Sind $(\mathcal{U},\mathcal{V})<(\mathcal{W},\mathcal{Z})$ zwei Überdeckungen in $Cov(X,A)$, so heißt eine Funktion $p\colon(\mathcal{W},\mathcal{Z})\rightarrow(\mathcal{U},\mathcal{V})$ Projektion, falls $p(W)\supset W$ für alle $W\in\mathcal{W}$ gilt. Aufgefasst als Eckenabbildung definiert eine Projektion eine eindeutige simpliziale Abbildung $\nu(\mathcal{W},\mathcal{Z})\rightarrow\nu(\mathcal{U},\mathcal{V})$, welche wir auch mit $p$ bezeichnen. 
\end{defi}

Die zu einer Relation $(\mathcal{U},\mathcal{V})<(\mathcal{W},\mathcal{Z})$ in $Cov(X,A)$ gehörigen Projektionen haben eine sehr nützliche Eigenschaft:

\begin{lem} \label{projhom}
Ist eine Relation $(\mathcal{U},\mathcal{V})<(\mathcal{W},\mathcal{Z})$ in $Cov(X,A)$ gegeben, dann ist die von einer Projektion $(\mathcal{W},\mathcal{Z})\rightarrow(\mathcal{U},\mathcal{V})$ in Homologie induzierte Abbildung \[ H_{\ast}^{sing}(|\nu(\mathcal{W},\mathcal{Z})|)\rightarrow H_{\ast}^{sing}(|\nu(\mathcal{U},\mathcal{V})|)
\] von der Wahl der Projektion unabhängig. Einer Relation in $Cov(X,A)$ kann also eindeutig eine Abbildung zwischen den Homologiegruppen der zugehörigen Nerven zugeordnet werden. 
\end{lem}
\begin{bew}
Seien zwei Projektionen $p,p^{\prime}\colon(\mathcal{W},\mathcal{Z})\rightarrow(\mathcal{U},\mathcal{V})$ gegeben. Sei weiter $s$ ein Simplex in $\nu(\mathcal{W},\mathcal{Z})$ und $x\in\bigcap\limits_{\text{\tiny $W$ Ecke von } s}W$ ein Punkt. Das heißt, dass $x\in W$ ist für alle Ecken $W$ von $s$.\\ $p,p^{\prime}$ sind Projektionen, damit gilt also $x\in p(W)$ und $x\in p^{\prime}(W)$ für alle Ecken $W$ von $s$. Daraus folgt, dass
 \[
 x\in \left(\bigcap_{\text{\tiny $W$ Ecke von } s} p(W)\cap \bigcap_{\text{\tiny $W$ Ecke von  } s } p^\prime(W)\right)\neq\emptyset
\]ist.
 Der Simplex $s^{\prime}$, der von den Bildern der Ecken von $s$ unter den beiden Projektionen $p,p^{\prime}$ erzeugt wird, ist also ein Simplex in $\nu(\mathcal{U},\mathcal{V})$. Das heißt, dass $p(s)$ und $p^{\prime}(s)$ Seiten eines Simplex $s^{\prime}$ in $\nu(\mathcal{U},\mathcal{V})$ sind. (Falls $s$ in $\nu_A\mathcal{Z}$ ist, wähle $x\in A\cap\bigcap\limits_{\text{\tiny $W$ Ecke von }s}W$. Wie eben zeigt man, dass $s^{\prime}$ in $\nu_A\mathcal{V}$ liegt.) Simpliziale Abbildungen mit dieser Eigenschaft, dass die jeweiligen Bilder eines Simplex Seiten desselben Simplexes im Bildkomplex sind, heißen zueinander "`benachbart"'. Das Besondere ist, dass die zugehörigen Abbildungen zwischen den geometrischen Realisierungen homotop sind, denn $h\colon|\nu(\mathcal{W},\mathcal{Z})|\times I\rightarrow|\nu(\mathcal{U},\mathcal{V})|$ definiert durch $h(x,t)=(1-t)p(x)+tp^{\prime}(x)$ ist eine Homotopie zwischen $p$ und $p^{\prime}$ (aufgefasst als stetige Abbildungen zwischen den geometrischen Realisierungen.). Denn da $p(x)$ und $p^\prime(x)$ im selben Simplex $|s^\prime|$ liegen für alle $x$ aus einem Simplex $|s|$, ist die Abbildung $h$ stetig.\\ Das liefert auf Grund der Homotopieinvarianz der singulären Homologie die Behauptung.
\end{bew}

Nun ist es möglich, folgende Definition sinnvoll zu formulieren:

\begin{defi}[\v{C}ech-Homologie mit Überdeckungen]
Sei $(X,A)$ ein kompaktes topologisches Raumpaar und sei für $k\in\mathbb N$ ein Kofunktor \[J\colon Cov(X,A)\rightarrow\mathcal{A}\mathcal{B}\]  von der Kategorie der Überdeckungen in die Kategorie der abelschen Gruppen definiert durch\[
(\mathcal{U},\mathcal{V})\mapsto H_{k}^{sing}(|\nu(\mathcal{U},\mathcal{V})|);\]\[ \left((\mathcal{U},\mathcal{V})<(\mathcal{W},\mathcal{Z})\right)\mapsto\left(H_{k}^{sing}(p)\colon H_{k}^{sing}(|\nu(\mathcal{W},\mathcal{Z})|)\rightarrow H_{k}^{sing}(|\nu(\mathcal{U},\mathcal{V})|)\right).\] 
Dann ist die $k$-te relative \v{C}ech-Homologiegruppe des Raumpaares $(X,A)$ definiert durch \[
\check{H}_{k}(X,A)=lim (J).\]
\end{defi}

Der beschriebene Kofunktor ist ein inverses System in die Kategorie der abelschen Gruppen, dessen Limes also nach Satz \ref{invlim} existiert. Natürlich kann auch insbesondere die absolute \v{C}ech-Homologiegruppe eines kompakten topologischen Raumes $X$ betrachtet werden, indem $(X,A)=(X,\emptyset)$ gesetzt wird.

\subsection{\v{C}ech-Homologie (Variante 2)}

In Definition $\ref{umgkat}$ haben wir die Kategorie der Umgebungspaare $\Omega_{(X,A)}$ bereits kennen gelernt. Für die alternative Definition der \v{C}ech-Homologie wird sie erneut benötigt:

\begin{defi}[\v{C}ech-Homologie mit Umgebungen]
Sei $(X,A)$ ein kompaktes Raumpaar in einem kompakten euklidischen Umgebungsretrakt $\mathrm E$. Definiere für $k\in \mathbb N$ den Kofunktor \[K\colon\Omega_{(X,A)}\rightarrow\mathcal{AB}\] von der Kategorie der offenen Umgebungspaare von $(X,A)$ in $\mathrm E$ in die Kategorie der abelschen Gruppen durch\[\begin{array}{rll}
(U,V)&\mapsto& H_{k}^{sing}(U,V) \cr 
\left((U,V)<(W,Z)\right) &\mapsto & \left(H_{k}^{sing}(i)\colon H_{k}^{sing}(W,Z)\rightarrow H_{k}^{sing}(U,V)\right),\cr
\end{array}
\]wobei $i\colon (W,Z) \hookrightarrow (U,V)$ die Inklusion bezeichne.
Dann ist die $k-$te relative \v{C}ech-Homologiegruppe mit Umgebungen definiert durch:\[
\check{H}_{k}^{\mathpzc{U}}(X,A)=lim(K).
\]
\end{defi} 

Auch hier gilt wegen Satz \ref{invlim}, dass der Limes des inversen Systems, das durch den Funktor $K$ definiert wird, existiert und dass insbesondere auch die absoluten \v{C}ech-Homologiegruppen eines kompakten Raumes $X$ betrachtet werden können.

\section{Isomorphie - der absolute Fall}

Da wir Satz \ref{1.9} benutzen wollen, um die Isomorphie der beiden Definitionen zu zeigen, müssen wir uns noch um die (starke) Kofinalität eines Funktors kümmern. Als erstes müssen wir uns dafür eine Teilmenge der Überdeckungen $Cov(X)$ näher anschauen, auf der der Funktor definiert sein wird. Diese Teilmenge ist die der nummerierten Überdeckungen, welche in \cite{MR1335915} eingeführt werden. Auch die Definition des Nummerierungsfunktors und der Beweis seiner Kofinalität sind an \cite{MR1335915} angelehnt.

\begin{defi}[Nummerierte Überdeckungen]
Definiere $Cov^n(X)\subset Cov(X)$ als die Menge der offenen Überdeckungen $\mathcal{U}$ eines topologischen Raumes $X$, für die eine lokal endliche Partition der Eins $\pi=\left\{\pi_U\right\}_{U\in\mathcal{U}}$ existiert ( also stetige Abbildungen $\pi_U\colon X\rightarrow [0,1]$ mit $\sum_U\pi_U(x)=1$ für alle $x\in X$ ), so dass $\pi_U^{-1}(0,1]\subset U$ für alle $U\in\mathcal{U}$ gilt. Wir nennen $\pi$ eine Nummerierung von $\mathcal{U}$ und $\mathcal{U}$ heißt dann nummeriert.
\end{defi} 

Genau wie $Cov(X)$ ist $Cov^n(X)$ teilweise geordnet durch Verfeinerung und kann ebenso als Kategorie aufgefasst werden. Der folgende Beweis folgt im Wesentlichen \cite{MR0500780}(Thm 5.1.9).

\begin{lem}
Sei $X$ lokal kompakte Teilmenge eines euklidischen Umgebungsretraktes $\mathrm E$. Dann sind die Mengen $Cov^n(X)$ und $Cov(X)$ gleich.
\end{lem}
\begin{bew}
Zu zeigen ist, dass jede offene Überdeckung von $X$ eine Nummerierung besitzt. $X$ ist als lokal kompakter Teilraum eines euklidischen Umgebungsretraktes insbesondere parakompakt.\footnote{Siehe zum Beispiel \cite{MR1016814}(Theorem 2-66).} Das bedeutet, dass es für jede offene Überdeckung des Raumes eine offene, lokal endliche Verfeinerung gibt. Es ist möglich zu zeigen, dass für solche Räume jede offene Überdeckung eine nummerierte Überdeckung ist:\\ 
Sei $\mathcal{U}$ in $Cov(X)$ eine offene Überdeckung. Um die Teilung der Eins zu bekommen, wird zunächst eine abgeschlossene, lokal endliche Überdeckung, die eine Verfeinerung von $\mathcal{U}$ ist, benötigt. Die Konstruktion dieser funktioniert so: Da $X$ als Teilmenge eines euklidischen Umgebungsretraktes regulär ist, gibt es eine offene Überdeckung $\mathcal{W}$ von $X$ , so dass $\left\{\overline{W}, W \in\mathcal{W}\right\}$ eine Verfeinerung von $\mathcal{U}$ ist.\\
Wegen der Parakompaktheit ist es möglich, $\mathcal W$ lokal endlich zu verfeinern. Wählt man so eine lokal endliche Verfeinerung $\mathcal{V}$ der Überdeckung $\mathcal{W}$ und für alle $V\in \mathcal{V}$ ein $U_V\in\mathcal{U}$, so dass $\overline{V}\subset U_V$ gilt, so ist $F_U=\bigcup\limits_{U_V= U}{\overline{V}}$ eine abgeschlossene lokal endliche Überdeckung von $X$, welche per Definition eine Verfeinerung von $\mathcal{U}$ ist. Da die Mengen der Überdeckung abgeschlossen sind, gibt es für alle $U\in \mathcal{U}$ nach Urysohns Lemma\footnote{Siehe zum Beispiel \cite{MR1224675}(Lemma 10.2).} eine stetige Funktion \[g_U\colon X\rightarrow [0,1] \text{ , so dass } g_U(x)=0 \text{ für } x\in X\setminus U \text{ und } g_U=1 \text{ für } x\in F_U .\] 
Definiert man $g(x)=\sum\limits_{U\in\mathcal{U}}{g_U(x)}$, so liefert das wegen der lokalen Endlichkeit der Partition $\left\{F_U\right\}_{U\in\mathcal{U}}$ eine stetige Funktion $g\colon X\rightarrow \mathbb{R}$. Nun kann eine lokal endliche Partition der Eins $\left\{\pi_U:X\rightarrow [0,1]\right\}_{U\in\mathcal{U}}$ durch $\pi_U=g_U/g$ definiert werden. Diese erfüllt per Definition die Bedingung $\pi_U^{-1}(0,1]\subset U$ für alle $U\in \mathcal{U}$.  \\
\end{bew}

\begin{bem}
In \cite{MR0500780} wird sogar gezeigt, dass Parakompaktheit und die Eigenschaft, dass jede offene Überdeckung nummeriert ist, äquivalent zueinander sind. Das wird hier aber nicht benötigt. 
\end{bem}

Wir sind nun bereit den gewünschten Funktor zu definieren.

\begin{lem}[Nummerierungs-Funktor]\label{num-funkt}
 Sei $X$ ein kompakter topologischer Teilraum eines kompakten euklidischen Umgebungsretraktes. Definiere wie im Folgenden beschrieben einen Funktor $\Theta\colon Cov^n(X)=Cov(X)\rightarrow \Lambda_{X}$, wobei $\Lambda_{X}^{op}=\mathscr{P}^{X}$ ist. Nehme hierfür die folgende Abbildungsvorschrift:\[\begin{array}{rcl}
\mathcal{U}& \mapsto & \left[\pi\colon X\rightarrow |\nu\mathcal{U}|\right]\cr
(\mathcal{V}<\mathcal{U})&\mapsto & \left[p\colon|\nu\mathcal{U}|\rightarrow|\nu\mathcal{V}|\right],\end{array}
\] wobei $p$ die in \ref{proj} beschriebene, bis auf Homotopie eindeutige Projektion bezeichnet und $\pi\colon X\rightarrow |\nu\mathcal{U}|$ dadurch definiert wird, dass es $x\in X$ auf den Punkt abbildet, dessen baryzentrische Koordinaten  die $\left\{\pi_U(x)\right\}_{U\in\mathcal U}$ sind. ($\left\{\pi_U\right\}_{U\in\mathcal{U}}$  ist eine Nummerierung von $\mathcal{U}$.)\\
Dieser Funktor wird auch Nummerierungs-Funktor genannt.
\end{lem}

\begin{bew}
Die Abbildung $\pi\colon X\rightarrow |\nu\mathcal U|$ ist stetig, da sie in der starken Topologie\footnote{Vgl. Bemerkung \ref{statop} auf Seite \pageref{statop}.} stetig ist  und wir auf Grund der Kompaktheit von $X$ immer endliche Teilkomplexe von $|\nu\mathcal U|$ betrachten können.\\
Weiter hängt ihr Homotopietyp nicht von der Wahl einer Nummerierung $\left\{\pi_U\right\}_{U\in\mathcal{U}}$ ab, denn ist $\left\{\rho_U\right\}_{U\in\mathcal{U}}$ eine zweite Wahl, so ist $(1-t)\pi+t\rho,  0\leq t\leq1,$ eine Homotopie. Die Stetigkeit folgt wie zuvor aus der Stetigkeit in der starken Topologie.\\
Sei $\mathcal V<\mathcal{U}$ eine Relation in $Cov^n(X,A)$ und sei $\pi\colon X\rightarrow|\nu\mathcal U|$ eine zu $\Theta(\mathcal{U})$ gehörige Abbildung. Wir wissen, dass die Homotopieklasse einer Projektion $p\colon|\nu\mathcal{U}|\rightarrow|\nu\mathcal{V}|$ eindeutig bestimmt ist.\footnote{Vgl. Lemma \ref{projhom} auf  Seite \pageref{projhom}} Damit diese aber ein Morphismus in $\Lambda_{X}$ ist, muss es eine zu $\Theta(\mathcal{V})$  gehörige Abbildung $\phi$ geben, so dass $p\pi\simeq \phi$ gilt. Wähle $\left\{\phi_V\right\}_{V\in\mathcal{V}}$ folgendermaßen aus: $\phi_V(x)=\sum\limits_{p[U]=[V]}{\pi_U(x)}$. Dies ist eine Nummerierung, da die $\pi_U$ eine Nummerierung sind und $\mathcal{U}$ eine Verfeinerung von $\mathcal{V}$ ist. Sei $x\in\nu\mathcal{U}$ mit baryzentrischen Koordinaten $x_U$ in der Form $x=\sum\limits_{U\in\mathcal{U}}{x_U[U]}$ gegeben. Es ist \[
p(x)=\sum\limits_{V\in\mathcal{V}}{x_V[V]} \text{, wobei } x_V=\sum\limits_{p(U)=V}{x_U}.
\] Daher gilt $\phi\simeq p\pi$.
\end{bew}

\begin{satz}
Der Nummerierungsfunktor ist stark kofinal.
\end{satz}

\begin{bew} Zuerst muss wie immer gezeigt werden, dass $\Theta$ schwach kofinal ist. Sei hierfür $[\xi\colon X\rightarrow P_\xi]\in\Lambda_{X}$ gegeben. Es ist die Existenz einer nummerierten, offenen Überdeckung $\mathcal{U}$ von  $X$ und einer Homotopieklasse von Abbildungen $f\colon|\nu\mathcal{U}|\rightarrow P_\xi$ zu zeigen, so dass $\xi\simeq f\circ p$ gilt, wobei $p\colon X\rightarrow |\nu\mathcal U|$ die Abbildung sein soll, die $\Theta(\mathcal{U})$ definiert.\\
Gegeben eine Triangulierung von $P_\xi$ mit Eckenmenge $J$, bezeichne man die zugehörigen baryzentrischen Koordinaten mit  $\beta_j\colon P_\xi\rightarrow[0,1] ,\, j\in J$.\\
Wähle nun $\mathcal{U}$ als die Mengen $(\beta_j\xi)^{-1}(0,1]$ mit $j\in J$. Weiter ordne jedem $U\in\mathcal{U}$ den Index $j(U)\in J$ zu, so dass $U=(\beta_j\xi)^{-1}(0,1]$ ist, und definiere eine Nummerierung von $\mathcal{U}$ durch $\left\{\pi_U=\beta_{j(U)}\xi\right\}$. Für $x\in X$ ist $\sum\limits_{U\in\mathcal{U}}{\pi_U(x)}=1$, da die $\beta_j$ die baryzentrischen Koordinaten sind, und es gilt $\pi_U^{-1}(0,1]=U$ per Definition für alle $U\in\mathcal{U}$. Auch dass die Nummerierung lokal endlich ist, folgt daraus, dass sie durch die baryzentrischen Koordinaten definiert ist.\\
$\Theta(\mathcal{U})$ wird dann durch die Abbildung $p\colon X \rightarrow\nu\mathcal{U}$ mit \[p(x)=\sum\limits_{U\in\mathcal{U}}{\beta_{j(U)}\xi(x)\left[(\beta_{j(U)}\xi)^{-1}(0,1]\right]}\] bestimmt. Definiert man die gesuchte Homotopieklasse von Abbildungen eindeutig durch die simpliziale Abbildung $f\colon\nu\mathcal{U}\rightarrow P_\xi$ mit $f(U)=j(U)$ für alle $U\in\mathcal{U}$, so gilt $f\circ p\simeq\xi$.

Für die starke Kofinalität seien zu einem Objekt $[\xi\colon X\rightarrow R_\xi]$ aus $\Lambda_{X}$ zwei nummerierte, offene Überdeckungen $\mathcal{U}$ und $\mathcal{V}$ von $X$ und Homotopieklassen von Abbildungen $[f\colon|\nu\mathcal{U}|\rightarrow P_\xi]$ und $[g\colon|\nu\mathcal{V}|\rightarrow P_\xi]$ gegeben, so dass \[
f\circ\pi_{\mathcal{U}}\simeq\xi\simeq g\circ\pi_{\mathcal{V}}
\] gilt, wobei $\pi_{\mathcal{U}}$ und $\pi_{\mathcal{V}}$ jeweils die von $\Theta(\mathcal{U})$ und $\Theta(\mathcal{V})$ bestimmten Abbildungen sind. Im Diagramm sieht das (bis auf Homotopie) so aus: \[
\begin{xy}
		\xymatrix{ & |\nu\mathcal{U}| \ar[rd]^{f} & \cr
							X \ar[rr]^{\xi}\ar[ru]^{\pi_{\mathcal{U}}} \ar[rd]_{\pi_{\mathcal{V}}} & &P_\xi \cr
							 & |\nu\mathcal{V}| \ar[ru]_{g} & \cr
		}
\end{xy}
\]
Es ist nun zu zeigen, dass dieses Diagramm (bis auf Homotopie) wie folgt vervollständigt werden kann:\[\begin{xy}\xymatrix{
& |\nu\mathcal{U}| \ar[rd]^{f} & \cr
X \ar[r]^{\pi}\ar[ru]^{\pi_{\mathcal{U}}} \ar[rd]_{\pi_{\mathcal{V}}} & |\nu\mathcal{W}| \ar[u] \ar[d] &P_\xi \cr
& |\nu\mathcal{V}| \ar[ru]_{g} & \cr
}\end{xy}\]
Wir wissen schon, dass $\Theta$ schwach kofinal ist. Es reicht also, das Diagramm anstatt mit einem Nerv $\nu\mathcal{W}$ mit einem Polyeder zu vervollständigen, denn dann faktorisiert die zugehörige Abbildung in das Polyeder über einen Nerv (bis auf Homotopie).\\
Sei $h\colon[0,1]\times(X,A)\rightarrow P_\xi$ die Homotopie zwischen $f\circ\pi_{\mathcal{U}}$ und $g\circ\pi_{\mathcal{V}}$.
Defniere \[
R=\left\{(x,y,\omega)\in\nu\mathcal{U}\times\nu\mathcal{V}\times P_\xi^{[0,1]} \mid \omega(0)=f(x),\omega(1)=g(y)\right\}
\] und eine Abbildung $\varphi\colon X\rightarrow R$ durch\[
\varphi(x)=(\pi_{\mathcal{U}}(x),\pi_{\mathcal{V}}(x),\omega_x)\text{ , wobei } \omega_x \text{ durch } \omega_x(i)=h_i(x) ;\, i\in[0,1] \text{ definiert ist.} 
\]
Es resultiert das folgende (bis auf Homotopie) kommutative Diagramm:\[\begin{xy}\xymatrix{
& \nu\mathcal{U} \ar[rd]^{f} & \cr
X \ar[r]^{\varphi}\ar[ru]^{\pi_{\mathcal{U}}} \ar[rd]_{\pi_{\mathcal{V}}} & R \ar[u]_{pr_{1}} \ar[d]^{pr_2} &P_\xi \cr
& \nu\mathcal{V} \ar[ru]_{g} & \cr
}\end{xy}\]
Zu zeigen bleibt also nur, dass $R$ homotopieäquivalent zu einem Polyeder ist. Da das Produkt von Polyedern wieder homotopieäquivalent zu einem Polyeder ist und weiter die beiden ersten Faktoren, die $R$ ausmachen, Polyeder sind, bleibt der dritte Faktor zu betrachten. Dieser ist von der Form $(P_\xi;im (f), im (g))^{([0,1];[0],[1])}$, wobei $P_\xi$ ein Polyeder ist und die Bilder $im(f)$ beziehungsweise $im(g)$ der gegebenen Abbildungen durch simpliziale Approximation der Abbildungen $f$ und $g$ als Unterkomplexe von $P_\xi$ angesehen werden können. Von Räumen dieser Art wird in \cite{MR0100267}(Thm. 3) gezeigt, dass sie homotopieäquivalent zu Polyedern sind.
\end{bew}

Nun ist es möglich, die Isomorphie der beiden absoluten Definitionen der \v{C}ech-Homologie zu beweisen:

\begin{satz}\label{iso}
Sei $X$ eine kompakte Teilmenge eines kompakten euklidischen Umgebungsretraktes. Dann ist \[
\check{H}^{\mathcal{U}}_{k}(X)\cong\check{H}_{k}(X)
\] für $k \in \mathbb N$.
\end{satz}

\begin{bew}
Die angegebene Isomorphie zu zeigen, bedeutet nichts anderes, als die Isomorphie zweier Limes zu zeigen, genauer gesagt die Isomorphie von $lim(J)$ und $lim(K)$, wobei \[
J\colon Cov(X)  \rightarrow  \mathcal{AB}  \quad \text{ durch } \]
\[\begin{array}{rll}
\mathcal{U} & \mapsto & H_{k}^{sing}(|\nu\mathcal{U}|) \cr \left(\mathcal{U}<\mathcal{V}\right) & \mapsto & \left(H_{k}^{sing}(p)\colon H_{k}^{sing}(|\nu\mathcal{V}|)\rightarrow H_{k}^{sing}(|\nu\mathcal{U}|)\right)  \cr
\end{array}\] und\[
K\colon\Omega_X  \rightarrow  \mathcal{AB} \quad \text{ durch } \]
\[\begin{array}{rll}
U & \mapsto & H_{k}^{sing}(U)  \cr 
\left(U<V\right) & \mapsto & \left(H_{k}^{sing}(i)\colon H_{k}^{sing}(V)\rightarrow H_{k}^{sing}(U)\right)  \cr
\end{array}\] gegeben sind.\\
Betrachte den Hilfsfunktor $d\colon\Lambda_{X}\rightarrow \mathcal{AB}$, der durch \[
\left[\xi\colon X\rightarrow P_\xi\right]\mapsto H^{sing}_{k}(P) \]\[ \quad \left[\alpha\colon P \rightarrow Q\right]\mapsto \left(H_{k}^{sing}(\alpha)\colon H_{k}^{sing}(P)\rightarrow H_{k}^{sing}(Q)\right)
\] definiert wird.\\
Wir können zeigen, dass sowohl $lim(J)$ als auch $lim(K)$ isomorph zu $lim(d)$ sind, und das liefert die Behauptung.\\
Zunächst rufe man sich den stark kofinalen Funktor $\Theta\colon Cov^n(X)=Cov(X)\rightarrow \Lambda_{X}$ (den Nummerierungsfunktor) in Erinnerung und stelle fest, dass
\begin{eqnarray*} 
J= d\circ \Theta
\end{eqnarray*} ist.
Mit Satz \ref{1.9} folgt dann \[lim(J)\cong lim(d),\] da $\Theta$ stark kofinal ist.
In Satz \ref{2.4} wurde die starke Kofinalität des Funktors $\vartheta\colon\Omega_X\rightarrow \Lambda_{X}$ mit \[ 
U\mapsto \left[X\hookrightarrow U\right]; \qquad (U<V)\mapsto \left[V\hookrightarrow U\right]
\] bewiesen. Da \[K=d\circ\vartheta\] ist, folgt erneut mit Satz \ref{1.9}, dass gilt: \[lim(K)\cong lim(d).\] 
\end{bew}

\section{Der relative Fall}

Die klassische Variante der \v{C}ech-Homologie, die über die Nerven der Überdeckungen de\-fi\-niert ist, hängt im Gegensatz zu der Definition, die die Umgebungen benutzt, nicht vom umgebenden Raum ab. Die Isomorphie der beiden liefert also insbesondere das Resultat, dass die Umgebungs-\v{C}ech-Homologie des Raumes $X$ nur von diesem abhängt.\\ 
Für kompakte Raumpaare $(X,A)$ gilt dann\[
\check{H}_k(X,A)\cong lim\left\{H^{sing}_k(X,V) | V \text{ offene Umgebung von $A$ }\right\}.
\]
Um diese Unabhängigkeit vom umgebenden Raum im relativen Fall zu bekommen, kann ein Standard-Trick angewendet werden, der die relative Homologie auf die absolute zu\-rück\-führt. Betrachte zu diesem Zweck für ein kompaktes Raumpaar $(X,A)$ eines kompakten eu\-kli\-disch\-en Umgebungsretraktes $\mathrm E$ als Hilfsraum $X\cup CA$, wobei $CA=(A\times[0,1])/(A\times \left\{1\right\})$ den Kegel von $A$ bezeichnet und wir $A\times\left\{0\right\}\subset CA$ in $X\cup CA$  mit $A\subset X$ identifizieren. Dies ist der sogenannte Abbildungskegel der Inklusionsabbildung $A\hookrightarrow X$. Er liegt kompakt im Kegel des umgebenden euklidischen Umgebungsretraktes, welcher wegen der Kompaktheit des ursprünglichen euklidischen Umgebungsretraktes selbst ein kompaktes euklidisches Umgebungsretrakt ist\footnote{Siehe zum Beispiel \cite{MR1224675} Korollar E.7.}.\\ 
Definiere nun eine alternative Umgebungs-\v{C}ech-Homologie von $X\cup CA\subset C\mathrm E$, in der nicht alle offenen Umgebungen bei der Bildung des Limes berücksichtigt werden, sondern nur solche von der Form $\bf{U}$$ \cup CV$. Hierbei bezeichnet $\bf U$ eine offene Umgebung von $X$ in $U\times [0,\epsilon]$ für ein $\epsilon>0$, wobei $U$ eine offene Umgebung von $X$ in $\mathrm E$ ist. Desweiteren sei $V$ eine offene Umgebung von $A$ in $\mathrm E$.\\
Es soll also der Limes über die singuläre Homologie solcher Umgebungen betrachtet werden. Diese Homologie ist im Wesentlichen die des Paares $(\bf U$ $\cup CV,\ast)$ (reduzierte Homologie), wobei $\ast$ die Kegelspitze bezeichnet. Da der Kegel sich zu seiner Spitze zusammenziehen lässt, kann dafür auch die singuläre Homologie von $(\bf U$ $\cup CV,CV)$ betrachtet werden und wiederum wegen Homotopieäquivalenz der Räume diese durch die singuläre Homologie des Paares $(U,V)$ ersetzt werden. 
Das ist aber die Definition der \v{C}ech-Homologiegruppe des Paares $(X,A)$ mit Umgebungen.\\
Wären die speziellen Umgebungen von $X\cup CA$, wie sie oben beschrieben werden, stark kofinal in allen offenen Umgebungen (was per Definition gleichbedeutend damit ist, dass der entsprechende Inklusionsfunktor stark kofinal ist), so wäre der betrachtete Limes isomorph zu dem über alle offenen Umgebungen. Von letzterem Limes haben wir aber im vorherigen Abschnitt durch die Isomorphie zur klassischen \v{C}ech-Homologie gezeigt, dass er nur vom Raum $X\cup CA$ abhängt, also nur vom Paar $(X,A)$. Daher wäre dann auch die Paar-Definition der Umgebungs-\v{C}ech-Homologie nur vom Paar $(X,A)$ abhängig.\\
Um klar zu machen, dass die speziellen Umgebungen des Raumes $X\cup CA$ stark kofinal in allen offenen Umgebungen enthalten sind, bemerken wir zunächst, dass es auf Grund der Tatsache, dass die Menge der offenen Umgebungen gerichtet ist, ausreicht, schwache Kofinalität zu zeigen. Denn in diesem Fall folgt die starke Kofinalität aus der schwachen.
Es ist also zu zeigen, dass es für eine beliebige offene Umgebung von $X\cup CA$ in $C\mathrm E$ eine spezielle Umgebung gibt, die darin enthalten ist. Sei dafür eine beliebige, offene Umgebung $W$ von $X\cup CA$ in $C\mathrm E$ gegeben. Für alle $x\in X$ existiert eine Umgebung $U_x$ in $\mathrm E$, so dass $U_x\times [0,\epsilon_x]$ noch in $W$ enthalten ist.\\ Da $X$ kompakt ist, kann es mit endlich vielen der $U_x$ überdeckt werden. Seien endlich viele der $U_x$ ausgewählt, die $X$ überdecken, und sei $\delta$ definiert als das Minimum über die zugehörigen $\epsilon_x$, so ist \[\bigcup_{\text{\tiny ausgewählte }x} U_x \times [0,\delta)\] das gewünschte $\bf{U}$.\\
Ebenso gibt es für alle $a\in A$ eine Umgebung $U_a$, so dass $U_a\times [0,1]\subset W\subset C\mathrm E$ gilt. Dann ist $CA\subset CV\subset W$, wo $\bigcup_{a\in A} U_a=V$ gewählt ist. Insgesamt gilt für die hier gewählten $\bf{U}$ und $V$, dass $\bf{U}$$\cup CV\subset W$ ist. Die speziellen Umgebungen sind also stark kofinal in allen Umgebungen.

\bigskip

Zusätzlich gibt es den Ansatz, die relative Isomorphie direkt mit Satz \ref{1.9} zu beweisen. Dafür ist es nötig, die Definitionen aus Satz \ref{2.4} und Lemma \ref{num-funkt} zu Raumpaarversionen von stark kofinalen Funktoren zu erweitern. Obgleich dieser Ansatz hier nicht weiter verfolgt werden soll, sei angemerkt, dass auch auf dem Weg dorthin schon interessante Ergebnisse zu erzielen sind. Ein solches, die zur Identität homotopen Abbildungen zwischen Polyederpaaren betreffend, soll in diesem Abschnitt hergeleitet werden. Hierzu werden als erstes erneut einige simpliziale Grundsteine gelegt, die \cite{MR666554} entstammen. Anschließend ermöglichen dann homotopietheoretische Ergebnisse den gewünschten Beweis.

\bigskip

Ein unerlässliches Werkzeug im Umgang mit Simplizialkomlexen sind Unterteilungen solcher. Insbesondere die baryzentrische Unterteilung ist für viele Argumente hilfreich, in denen "`kleine"' Simplizes benötigt werden.

\begin{defi}[Unterteilung]
Sei $K$ ein Simplizialkomplex. Eine Unterteilung von $K$ ist ein Simplizialkomplex $K^\prime$, so dass gilt:
\begin{itemize}
	\item Die Ecken von $K^\prime$ sind Punkte von $|K|$.
	\item Ist $s^\prime$ ein Simplex von $K^\prime$, so gibt es ein Simplex $s$ von $K$, so dass $s^\prime\subset |s|$ ist, das heißt $s^\prime$ ist eine endliche nichtleere Teilmenge von $|s|$.
	\item Die lineare Abbildung $|K^\prime|\rightarrow|K|$, die jede Ecke von $K^\prime$ auf den zugehörigen Punkt in $|K|$ abbildet, ist ein Homöomorphismus.
\end{itemize}
\end{defi}

\begin{defi}[Baryzentrische Unterteilung]
Sei $K$ ein Simplizialkomplex. Definiere den Schwerpunkt (Baryzenter) $b(s)$ eines Simplexes $s=\left\{v_0, v_1,\ldots, v_q\right\}$ in $K$ als den Punkt\[
b(s)=\sum\limits_{0\leq i \leq q}{\frac{1}{q+1}v_i} 
\]
und den Simplizialkomplex $sd K$ als denjenigen Simplizialkomplex, dessen Ecken die Schwerpunkte der Simplizes von $K$ sind und dessen Simplizes endliche, nichtleere Teilmengen von Schwerpunkten von Simplizes sind, welche durch die Relation "`ist Seite von"' vollständig geordnet sind.\\
Die Simplizes von $sd K$ sind also endliche Mengen $\left\{b(s_0),\ldots, b(s_q)\right\}$, so dass $s_{i-1}$ eine Seite von $s_i$ ist für alle $i=1,\ldots, q$.\\
Der Simplizialkomplex $sd K$ heißt baryzentrische Unterteilung von $K$.
\end{defi}

Der Beweis dafür, dass die baryzentrische Unterteilung eine Unterteilung ist, kann zum Beispiel in \cite{MR666554} (Thm 3.3.9) nachgelesen werden. Hier wird er ausgespart. Jedoch beweisen wir im nächsten Satz ein nützliches Werkzeug, um zu zeigen, dass ein Simplizialkomplex eine Unterteilung eines anderen Simplizialkomplexes ist.

\begin{satz}\label{endunt}
Seien $K$ und $K^\prime$ Simplizialkomplexe, die die ersten beiden Unterteilungsbedingungen erfüllen. Dann ist $K^\prime$ eine Unterteilung von $K$ genau dann, wenn für alle Simplizes $s$ aus $K$ die Menge $\left\{<s^\prime>\mid s^\prime\in K^\prime , <s^\prime>\subset<s>\right\}$ eine endliche Partition von $<s>$ ist.
\end{satz}

\begin{bew}
\begin{itemize}
  \item []
	\item ["`$\Rightarrow$"'] Wenn $K^\prime$ eine Unterteilung von $K$ ist, dann gilt: \begin{equation}
													\left\{s^\prime\mid s^\prime \in K^\prime\right\} \text{ ist eine Unterteilung von } |K^\prime|\approx|K|.
													\label{Un}
													\end{equation}
													Sei $s\in K$ und betrachte $<s>\cap<s^\prime>$ für $s^\prime\in K^\prime$.\\ 
													Es ist entweder $<s>\cap<s^\prime>=\emptyset$ oder $<s>\cap<s^\prime>\neq\emptyset$. Im zweiten Fall ist dann aber auch                             $|\tilde{s}|\cap<s>\neq\emptyset$ für das wegen der zweiten Unterteilungseigenschaft existierende $\tilde{s}\in K$ mit                              $s^\prime\subset |\tilde{s}|$. Daraus folgt, dass $s=\tilde{s}$, da die offenen Simplizes $K$               partitionieren. Insgesamt ist dann aber $<s^\prime>\subset<s>$, was verbunden mit (\ref{Un}) ergibt, dass\[
													\left\{<s^\prime>\mid s^\prime\in K^\prime ,<s^\prime>\subset<s>\right\}
													\] eine Partition von $<s>$ ist. Aus der Kompaktheit von $|s|$ folgt, dass diese endlich ist.                                    
	\item["`$\Leftarrow$"'] Da jedes Simplex von $K$ eine endliche Anzahl von Seiten hat, ist \[
													 K^\prime(s)=\left\{s^\prime\in K^\prime\mid\exists\text{ Seite } s_1\subset s \text{ mit}<s^\prime>\subset<s_1>\right\} 
													\] ein endlicher Unterkomplex von $K^\prime$. Gemeinsam mit der Voraussetzung liefert das, dass die lineare Abbildung                               $h_s\colon|K^\prime(s)|\rightarrow |s|$, die jede Ecke von $K^\prime(s)$ auf sich selbst abbildet, ein Homöomorphismus                              ist.\\ Also gibt es eine stetige Abbildung $g\colon|K|\rightarrow|K^\prime|$, so dass $g_{|_{|s|}}=h_s^{-1}$ für $s\in K$                           ist, welche eine Inverse der Abbildung $h\colon|K^\prime|\rightarrow|K|$ aus der dritten Unterteilungseigenschaft ist.
													Diese ist somit ein Homöomorphismus, was zu zeigen war.
\end{itemize}
\end{bew}

Da uns in diesem Abschnitt vor allem Paare von Simplizialkomplexen (und Polyedern) interessieren, sollten wir wissen, wie in Bezug auf Unterteilungen mit Unterkomplexen umgegangen werden kann. Das liefert uns das nächste Lemma:

\begin{lem}
Sei $K^{\prime}$ eine Unterteilung von $K$ und sei $L$ ein Unterkomplex von $K$. Dann gibt es einen eindeutigen Unterkomplex $L^{\prime}$ von $K^{\prime}$, der eine Unterteilung von $L$ ist.\\ Diese Unterteilung $L^\prime$ von $L$ heißt die durch $K^\prime$ induzierte Unterteilung von L und sie wird mit $K^{\prime}|L$ bezeichnet.
\end{lem}
\begin{bew}
Ist $L^\prime$ ein Unterkomplex von $K^\prime$, der eine Unterteilung von $L$ ist, so ist \linebreak[4] $ L^\prime=\left\{s^\prime\in K^\prime \mid <s^\prime> \subset |L| \right\}$, was die Eindeutigkeit von $L^\prime$ zeigt.\\
Um die Existenz von $L^\prime$ zu zeigen, zeigt man, dass $\left\{s^\prime\in K^\prime \mid <s^\prime> \subset |L|\right\}$ die Eigenschaften einer Unterteilung von $L$ besitzt. (Diese Menge ist ein Unterkomplex $L^\prime$ in $K^\prime$, da falls $<s^\prime>\subset |L|$ für ein Simplex $s^\prime$ in $K^\prime$ ist, dann gilt dies auch für alle seine Seiten.)\\
Es ist leicht zu sehen, dass $L^\prime$ die erste Unterteilungseigenschaft erfüllt. Die zweite Unterteilungseigenschaft folgt, da falls $<s^\prime>\subset|L|$ ist, es auch ein Simplex in $L$ gibt, so das $s^\prime$ darin liegt. \\
Um die dritte Unterteilungseigenschaft zu zeigen, reicht es wegen Satz \ref{endunt} zu zeigen, dass für alle $s$ aus $L$ die Menge $\left\{<s^\prime>\mid s^\prime\in L^\prime , <s^\prime>\subset<s>\right\}$ eine endliche Partition von $<s>$ ist.\\ Da wir wissen, dass $K^\prime$ eine Unterteilung von $K$ ist, ist wiederum wegen Satz \ref{endunt} für alle $s$ aus $L$ die Menge $\left\{<s^\prime>\mid s^\prime\in K^\prime , <s^\prime>\subset<s>\right\}$ eine endliche Partition von $<s>$. Nach der Definition von $L^\prime$ ist aber \[
\left\{<s^\prime>\mid s^\prime\in K^\prime , <s^\prime>\subset<s>\right\}=\left\{<s^\prime>\mid s^\prime\in L^\prime , <s^\prime>\subset<s>\right\},
\] woraus die Behauptung folgt.
\end{bew}

Eine nützliche Aussage für Paare von topologischen Räumen $(X,A)$ und deren Triangulierungen $((K,L),f)$ (das heißt $f\colon(X,A)\stackrel{\approx}{\rightarrow}(|K|,|L|)$) folgt direkt aus den Definitionen und dem vorherigen Lemma:

\begin{satz}\label{sub}
Ist $((K,L),f)$ eine Triangulierung von $(X,A)$ und $K^{\prime}$ eine Unterteilung von $K$, dann ist $((K^{\prime},K^{\prime}|L),f)$ auch eine Triangulierung von $(X,A)$.
\end{satz}

Später wird es von Bedeutung sein, einen Simplizialkomplex in Bereiche aufzuteilen. Um dann eine vernünftige Darstellung der Simplizes in Bezug auf diese Aufteilung zu erhalten, ist folgende Begriffsbildung entscheidend. Warum das so ist, lässt sich im anschließenden Lemma erkennen.

\begin{defi}
Ein Unterkomplex $L\subset K$ eines simplizialen Komplexes $K$ heißt voll, falls jedes Simplex von $K$, dessen Ecken alle in $L$ sind, selbst auch zu $L$ gehört.
\end{defi}

\begin{lem}\label{auft}
Sei $L$ ein voller Unterkomplex eines Simplizialkomplexes $K$ und sei $N$ der größte Unterkomplex von $K$, der disjunkt zu $L$ ist. Dann ist jedes Simplex von $K$ entweder in $N$ oder in $L$ oder von der Form $s ^{\prime}\cup s^{\prime\prime}$ für ein $s^{\prime}\in L$ und ein $s^{\prime\prime}\in N$. 
\end{lem}

\begin{bew}
Sei $s=\left\{v_0,v_1,\ldots,v_q\right\}$ ein Simplex von $K$. Dann ist entweder keine Ecke von $s$ in $L$ (also $s\in N$) oder alle Ecken von $s$ sind in $L$, was bedeutet, dass $s\in L$ ist, da $L$ voll ist, oder es ist möglich, die Ecken so zu nummerieren, dass $v_i\in L$ für $i\leq p$ und $v_i\notin L$ für $i>p$, wobei $0\leq p<q$. Im letzten Fall ist $s=s ^{\prime}\cup s^{\prime\prime}$, wobei $s^{\prime}=\left\{v_0,v_1,\ldots,v_p\right\}$ in $L$ ist, da $L$ voll ist und $s^{\prime\prime}=\left\{v_{p+1},\ldots,v_q\right\}$ in $N$ ist.
\end{bew}

Der Begriff eines vollen Unterkomplexes macht nicht zuletzt auch deswegen Sinn, da die baryzentrische Unterteilung auf natürliche Weise einen solchen liefert:

\begin{lem}\label{voll}
Es bezeichne $sd K$ die baryzentrische Unterteilung eines simplizialen Komplexes $K$. Dann gilt: Ist $L$ ein Unterkomplex von $K$, so ist $sd L$ ein voller Unterkomplex von $sd K$.
\end{lem}

\begin{bew}
Sei $\left\{b(s_0), b(s_1),\ldots, b(s_q)\right\}$ ein Simplex von  $sd K$, dessen Ecken alle zu $sd L$ gehören. Das heißt, dass $s_{i-1}$ eine Seite von $s_i$ ist für alle $i=1,2,\ldots,q$ und dass jedes $s_i$ in $L$ liegt. Also liegt auch $\left\{b(s_0), b(s_1),\ldots, b(s_q)\right\}$ in $sd L$.
\end{bew}

Homotopien erweitern zu können, liefert den Schlüssel für viele Beweise. Die Definiton der Homotopieerweiterungseigenschaft ist in diesem Zusammenhang fundamental.

\begin{defi}[Homotopieerweiterungseigenschaft]
Man sagt, ein topologisches Raumpaar $(X,A)$ habe die Homotopieerweiterungseigenschaft bezüglich eines topologischen Raumes $Y$, falls für eine gegebene Abbildung $g\colon X \rightarrow Y$ und eine gegebene Homotopie $G\colon A\times I\rightarrow Y$, so dass $G(x,0)=g(x)$ für $x\in A$ gilt, eine Abbildung $F\colon X\times I\rightarrow Y$ existiert, für die gilt: \[
F(x,0)=g(x) \text{ für } x\in X \text{ und } F|_{{ A\times I}} = G.
\]
\end{defi}

Anders ausgedrückt hat ein Raumpaar $(X,A)$ die Homotopieerweiterungseigenschaft bezüglich $Y$, wenn die partielle Homotopie $G\colon X\times \left\{0\right\}\cup A\times I\rightarrow Y$ zu einer Homotopie $F\colon X\times I\rightarrow Y $  erweitert werden kann.\\
Da wir als nächstes sehen wollen, dass es tatsächlich Räume gibt, die die Homotopieerweiterungseigenschaft besitzen, betrachten wir folgendes Lemma, in dem für einige Fälle die Gültigkeit der Homotopieerweiterungseigenschaft gesichert wird. Es ist \cite{MR0181977} entnommen und wird dort \textit{"`Dowker's Lemma"'} genannt.\\
Es enthält die Begriffsbildung eines abzählbar parakompakten Raumes. Diesen charakterisiert die Eigenschaft, dass jede abzählbare offene Überdeckung des Raumes eine lokal endliche Verfeinerung hat. Wir interessieren uns in diesem Zusammenhang nur für Polyeder, von denen bekannt ist, dass sie sogar parakompakt sind\footnote{Siehe zum Beispiel \cite{cw-top}(Thm 4.2).}, es also für jede offene Überdeckung eine lokal endliche Verfeinerung gibt.

\begin{lem}\label{dowker}
Sei $A$ eine abgeschlossene Teilmenge eines normalen, abzählbar parakompakten Raumes $X$ und $Y$ ein beliebiger topologischer Raum, und bezeichne mit $T$ die abgeschlossene Teilmenge $T:=X\times\left\{0\right\}\cup A\times I$  des Raumes $X\times I$. Hat eine Abbildung $f\colon T\rightarrow Y$ eine Erweiterung über $(X\times\left\{0\right\})\cup U$, wobei $U$ eine offene Menge in $X\times I$ ist, die $A\times I$ enthält, so hat $f$ auch eine Erweiterung über $X\times I$.
\end{lem}
\begin{bew}
Sei $g\colon(X\times\left\{0\right\})\cup U\rightarrow Y$ eine Erweiterung von $f$, wobei $U$ eine offene Menge in $X\times I$ ist, die $A\times I$ enthält.
Da $X$ normal und abzählbar parakompakt ist, ist $X\times I$ normal\footnote{Siehe \cite{MR0043446}(Thm. 4).}. Daher ist es möglich, eine offene Umgebung $V$ von $A$ in $X$ zu finden, so dass $V\times I\subset U$ gilt. $A$ und $X\setminus V$ sind disjunkte, abgeschlossene Teilmengen eines normalen Raumes $X$, daher existiert nach Urysohn's Lemma \footnote{Siehe beispielsweise \cite{MR1224675}(Lemma 10.2).} eine Abbildung $e\colon X\rightarrow I$, so dass\[
e(x)=\left\{\begin{array}{ll} 1 &\text{ , falls } x\in A\cr
									 0 &\text{ , falls } x\in X\setminus V.
\end{array}\right.
\]
Definiere eine Abbildung $h\colon X\times I\rightarrow$ durch \[
h(x,t)=g\left(x,e\left(x\right)t\right)\]
für alle $x\in X$ und alle $t\in I$. Diese ist eine Erweiterung von $f$ über $X\times I$.
\end{bew}

Das finale Resultat ist nun mit Ideen aus ähnlichen Beweisen in \cite{MR666554} und \cite{MR0050886} erreichbar:

\begin{satz}\label{retr}
Sei $(R,Q)$ ein polyedrisches Paar. Dann existert eine offene Menge $B$, wobei $Q\subset B\subset R$ und eine Abbildung $\omega\colon(R,Q)\rightarrow(R,Q)$, so dass $\omega(B)\subset Q$ und $\omega\simeq id_R$ gilt.
\end{satz}

\begin{bew}
Wegen der Sätze \ref{sub} und \ref{voll} kann man ohne Beschränkung der Allgemeinheit davon ausgehen, dass $(R,Q)\approx(|K|,|L|)$ ist, wobei $L$ ein voller Unterkomplex des simplizialen Komplexes $K$ ist.\\
Der folgende Beweis soll so funktionieren, dass wir zeigen, dass es eine abgeschlossene Teilmenge von $|K|$ gibt, von welcher $|L|$ ein starkes Deformationsretrakt ist. $|K|$ zusammen mit dieser Teilmenge erfüllt dann die Homotopieerweiterungseigenschaft wegen Lemma \ref{dowker}. Schließlich wird $B$ als die gefundene Menge ohne ihren Rand definiert.\\
Es sei $N$ der größte Unterkomplex von $K$ disjunkt zu $L$. Wir können nun zunächst zeigen, dass $|L|$ ein starkes Deformationsretrakt von $|K|\setminus|N|$ ist:\\
Sei dafür $\alpha\in|K|\setminus|N|$. Dann ist wegen Satz \ref{auft} $\alpha\in |L|$ oder es existieren Ecken $v_0, v_1,\ldots, v_p\in L$ und Ecken $v_{p+1}, \ldots, v_q\in N$ mit $0\leq p$ und $p+1\leq q$, so dass\linebreak[4] $\alpha\in<v_0,\ldots,v_q>$. Im zweiten Fall gibt es eine Darstellung von $\alpha$ in baryzentrischen Koordinaten: $\alpha=\sum\limits_{0\leq i\leq q}{\alpha_i v_i}$, wobei $\alpha_i>0$ ist.\\
Definiere nun $a:=\sum\limits_{i=0}^{p}{\alpha_i}$, für welches dann $0<a<1$ gilt. Mit \[\alpha_i^{\prime}:=\frac{\alpha_i}{a}\text{ für } 0\leq i\leq p \text{ und } \alpha_i^{{\prime}{\prime}}:=\frac{\alpha_i}{(1-a)} \text{ für } p+1\leq i \leq q \text{ , sowie }\] \[\alpha^{\prime}:=\sum\limits_{i=0}^{p}{\alpha_i^{\prime}v_i}\in|L| \text{ und }      \alpha^{{\prime}{\prime}}:=\sum\limits_{i=p+1}^{q}\alpha_i^{{\prime}{\prime}}v_i\in|N| \text{ , gilt dann }\] $\alpha=a\alpha^{\prime}+(1-a)\alpha^{{\prime}{\prime}}.$\\
Definiere nun $r\colon |K|\setminus |N|\rightarrow|L|$ mit dieser Notation durch \[
\alpha\mapsto \left\{\begin{array}{ll} \alpha & \text{ , falls } \alpha \in |L|\cr
														\alpha^{\prime} & \text{ , sonst }.
							\end{array}\right.
\] 
Bezeichnet man mit $i\colon|L|\hookrightarrow|K|\setminus|N|$ die Inklusion, dann ist die folgende Homotopie $H\colon |K|\setminus|N|\times I\rightarrow |K|\setminus|N|$ eine starke Deformationsretraktion von $|K|\setminus|N|$  nach $|L|$ , also eine Homotopie $id_{|K|\setminus|N|}\simeq ir$ relativ zu |L|:\[
H(\alpha,t)=\left\{\begin{array}{ll}
												\alpha & \text{ für } \alpha \in |L|, t\in I\cr
												t\alpha^{\prime}+(1-t)\alpha & \text{ für } \alpha\in |K|\setminus(|N|\cup|L|), t\in I.
												\end{array}\right.
\]
Die Stetigkeit von $H$ folgt daraus, dass $H|_{|L|\times I}$ stetig ist und für alle Simplizes der Form $s^{\prime}\cup s^{{\prime}{\prime}}$, wobei $s^\prime\in L$ und $s^{{\prime}{\prime}}\in N$, $H|_{[|s^{\prime}\cup s^{{\prime}{\prime}}|\cap (|K|\setminus|N|)]\times I}$ stetig ist.\\
Betrachtet man nun die baryzentrische Unterteilung $sdK$ und hierin den größten Unterkomplex $N_1$, der disjunkt zu $L$ ist, so gilt $\overline{|K|\setminus|N_1|}\subset |K|\setminus|N|$, da $|N|$ echt in $|N_1|$ liegt. Weiter gilt für $\alpha\in \overline{|K|\setminus|N_1|}$, dass $H(\alpha,t)\in\overline{|K|\setminus|N_1|}$ ist für alle $t\in I$, denn die Homotopie $H$ schickt Punkte entweder auf sich selbst oder auf Punkte, die näher an $|L|$ liegen als sie selbst und daher erst recht nicht zu $|N_1|$ gehören. \\
Also definiert die Einschränkung von $H$ auch eine starke Deformationsretraktion von $\overline{|K|\setminus|N_1|}$ nach $|L|$. Für das Paar $(|K|,\overline{|K|\setminus|N_1|})$ gibt es also per Konstruktion eine offene Umgebung von $\overline{|K|\setminus|N_1|}\times I$, welche hier $|K|\setminus|N|\times I$ ist, so dass sich die partielle Homotopie auf $|K|\times \left\{0\right\}\cup |K|\setminus|N|\times I$ erweitern lässt. Mit Lemma \ref{dowker} hat $(|K|,\overline{|K|\setminus|N_1|})$ also die Homotopieerweiterungseigenschaft.\\ Definiere nun $B=|K|\setminus|N_1|$ und $\omega$ als die zur Identität homotope Abbildung, die man aus der Homotopieerweiterungseigenschaft bekommt. Damit folgt das Gewünschte.
\end{bew}

%% file: simp2.tex
\subsection{Homotopietheoretischer Einschub}

Im folgenden Abschnitt geht es um Homotopie, genauer gesagt vor allem um die Homotopieäquivalenz von Räumen. Sein Inhalt ist zum größten Teil \cite{MR1867354} entnommen. Wegen des vorbereitenden Charakters dieses Kapitels wurde versucht, es kurz und knapp  zu halten, weshalb einige Beweise weggelassen wurden.\\
Der erste nun folgende Satz ist ein altes Resultat, welches nicht direkt etwas mit Homotopie zu tun hat, allerdings wird es benötigt, um die Homotopieäquivalenz der in Satz \ref{2.4} definierten Umgebungen zu Polyedern zu zeigen. Auch unabhängig davon,  dass und wofür wir das Ergebnis benötigen, ist es ein interessantes Resultat, denn es liefert die Existenz einer Triangulierung für Mengen, die dies auf den ersten Blick nicht vermuten lassen.\\
Für den Beweis wird eine "`zum Rand hin immer kleiner"' werdende Folge von Unterteilungen eines Simplizialkomplexes benötigt. Der Beweis ist nachzulesen in \cite{MR0345087}(III.3.2.), wo er auch \textit{"`Satz von Runge"'} genannt wird.

\begin{satz}[Offene Teilmengen von Polyedern im $\mathbb R^n$]\label{off-rn}
Sei $P$ ein Polyeder, welches für ein $n\in \mathbb N$ linear in den $\mathbb R^n$ einbettbar ist. Dann ist jede offene Teilmenge $U$ des Polyeders $P$ ein (im Allgemeinen unendliches) Polyeder.
\end{satz}

Wie angekündigt müssen wir uns vor allem damit beschäftigen, wann Räume homotopieäquivalent zu Polyedern sind. Die folgende Definition der Dominanz von Räumen kann dabei insofern nützlich sein, als dass Räume dominiert von Polyedern homotopieäquivalent zu CW-Komplexen sind, welche wiederum homotopieäquivalent zu Simplizialkomplexen sind.\\
Bis zum Ende des Abschnitts können die folgenden Sätze und Definitionen in \cite{MR1867354} wiedergefunden werden.

\begin{defi}Ein topologischer Raum $Y$ heißt dominiert von einem Raum $X$, falls es Abbildungen $Y\stackrel{i}{\rightarrow}X\stackrel{r}{\rightarrow}Y$ mit $ri\simeq id_Y$ gibt.
\end{defi}

Um Teile des oben Behaupteten zu beweisen, benötigen wir die folgende Definition des Abbildungsteleskops, welches eine Verallgemeinerung des Abbildungszylinders darstellt.

\begin{defi}[Abbildungszylinder, Abbildungsteleskop]
Sei $f\colon X\rightarrow Y$ eine Abbildung zwischen topologischen Räumen $X$ und $Y$. Die Menge $M_f$ mit\[
M_f := {(X\times I) \amalg Y}/\sim \text{ , wobei } (x,1)\sim f(x)
\] wird der Abbildungszylinder von $f$ genannt.\\
Weiter definiere für eine Folge von Abbildungen $X_1\stackrel{f_1}{\rightarrow}X_2\stackrel{f_2}{\rightarrow}X_3\stackrel{f_3}{\rightarrow}\ldots$ das Abbildungsteleskop als Quotientenraum von $\amalg_{i}(X_i\times[i,i+1])$, in dem $(x,i+1)\in X_i\times[i,i+1]$ mit $(f(x),i+1)\in X_{i+1}\times[i+1,i+2]$ identifiziert wird.\\ 
Das Abbildungsteleskop ist also die Vereinigung der Abbildungszylinder $M_{f_i}$, wobei für alle $i$ die Kopie von $X_{i}$ in $M_{f_i}$ mit $X_i\subset M_{f_{i-1}}$ identifiziert wird. Es wird mit $T(f_1,f_2,\ldots)$ bezeichnet.
\end{defi}

Die nachfolgenden Eigenschaften des Abbildungsteleskops machen es zu einer nützlichen Begriffsbildung.

\begin{lem}\label{teleig}
Sei eine Folge von Abbildungen $X_1\stackrel{f_1}{\rightarrow}X_2\stackrel{f_2}{\rightarrow}X_3\stackrel{f_3}{\rightarrow}\ldots$ von topologischen Räumen gegeben. Dann gilt:
\begin{itemize}
	\item[1.] $T(f_1,f_2,\ldots)\simeq T(g_1,g_2,\ldots)$, falls $f_i\simeq g_i$ für alle $i$.
	\item[2.] $T(f_1,f_2,\ldots)\simeq T(f_2,f_3,\ldots)$.
	\item[3.] $T(f_1,f_2,\ldots)\simeq T(f_2f_1,f_4f_3,\ldots)$.
\end{itemize}
\end{lem}

\begin{bew}
Man stelle sich $T(f_1,f_2,\ldots)$ als den Raum $\amalg_i(X_i\times\left\{i\right\})$ vor, in den mit Hilfe der $f_i$ der Raum $\amalg_i(X_i\times[i,i+1])$ eingeklebt worden ist. Es ist auch möglich, sich $T(g_1,g_2,\ldots)$ auf diese Art und Weise vorzustellen, nur mit $f_i$ ersetzt durch $g_i$. Da diese Abbildungen homotop sind, folgt die erste Aussage.\\
Die zweite Aussage folgt, da jeder Punkt $(x,t)$ des ersten Abbildungszylinders $M_{f_{1}}$ entlang des Segments ${x}\times I\subset M_{f_1}$ zu seinem Endpunkt $f_1(x)\in X_2$ stetig verschoben, das erste Abbildungsteleskop also durch Deformation in das zweite überführt werden kann.\\ 
Betrachte für die dritte Aussage $T(f_1,f_2,\ldots)$ als die disjunkte Vereinigung der Abbildungszylinder $M_{f_{2i}}$, in die der Raum $\amalg_i(X_{2i-1}\times[2i-1,2i])$ mit Hilfe der $f_{2i-1}$ eingeklebt wurde. Nun kann für alle $i$ Folgendes durchgeführt werden:\\
Man verschiebe das Ankleben (mittels $f_{2i-1}$) von $X_{2i-1}\times [2i-1,2i]$  an $X_{2i}\subset M_{f_{2i}}$ an letzterem Abbildungszylinder entlang nach $X_{2i+1}$. Dadurch geht man von $M_{f_{2i-1}}\cup M_{f_{2i}}$ nach $M_{f_{2i}f_{2i-1}}\cup M_{f_{2i}}$ über. Vom letzten Raum aus gibt es eine Deformationsretraktion nach $M_{f_{2i}f_{2i-1}}$ (wie bei der zweiten Aussage).
Damit sind alle Behauptungen bewiesen.
\end{bew}

\begin{lem}\label{telcw}
Sei eine Folge von Abbildungen von CW-Komplexen $X_1\stackrel{f_1}{\rightarrow}X_2\stackrel{f_2}{\rightarrow}X_3\stackrel{f_3}{\rightarrow}\ldots$ gegeben. Sind die Abbildungen $f_i$ für alle $i$ zellulär, so ist das Abbildungsteleskop ein CW-Komplex.
\end{lem}
\begin{bew}
Die Aussage folgt daraus, dass Abbildungszylinder zellulärer Abbildungen CW-Komplexe sind. Sei $f\colon X\rightarrow Y$ eine zelluläre Abbildung zwischen CW-Komplexen, das heißt $f(X^n)\subset Y^n$ für alle $n\in \mathbb N$.\\
Um nun zu sehen, dass $M_f$ eine CW-Struktur besitzt, nehmen wir als zu Grunde liegende diskrete Menge $X^0\cup Y^0$ an. An diese kleben wir die für $X$ und $Y$ benötigten $1-$Zellen an, wie es die jeweiligen Strukturen vorschreiben. Zusätzlich kleben wir im ersten Schritt für jedes $x\in X^0$ eine $1$-Zelle zwischen $x$ und $f(x)\in Y^0$ an.\\ Im zweiten Schritt werden wie gehabt zunächst die $2-$Zellen aus $X$ und $Y$ wie in der jeweiligen Struktur vorgesehen und zusätzlich für jede $1-$Zelle $I$ aus $X^1$ eine $2-$Zelle angeklebt, und zwar entlang des "`Kreises"' bestehend aus $I, f(I)$ und den $1-$Zellen, die im ersten Schritt zusätzlich zwischen deren Randpunkten eingeklebt worden sind.\\ Allgemein kleben wir also im $i-$ten Schritt (i>1) innerhalb $X$ und $Y$ alles an, wie von deren Struktur vorgeschrieben, zusätzlich aber auch noch für jede $(i-1)$-Zelle $D^{i-1}$ aus $X^{i-1}$ eine $i-$Zelle entlang der $(i-1)$-Sphäre bestehend aus $D, f(D)$ und dem, was im vorherigen Schritt zwischen deren Rändern eingeklebt worden ist.\\
Eine solche Konstruktion kann simultan in allen Abbildungszylindern, die das Abbildungsteleskop bilden, durchgeführt werden und mündet dann in einer CW-Struktur auf $T(f_1,f_2,f_3,\ldots)$.
\end{bew}

Dank dieser Vorbereitungen können wir nun also die Dominanz von Räumen benutzen, um die Homotopieäquivalenz zu Simplizialkomplexen nachzuweisen.

\begin{satz}\label{domCW}
Ein topologischer Raum, der von einem CW-Komplex dominiert wird, ist homotopieäquivalent zu einem CW-Komplex.
\end{satz}

\begin{bew}
Es seien topologische Räume $Y$ und $X$ gegeben, so dass es Abbildungen\linebreak[4] $Y\stackrel{i}{\rightarrow}X\stackrel{r}{\rightarrow}Y$ mit $ri\simeq id_Y$ gibt.\\
 Mit den Abbildungsteleskopeigenschaften aus Lemma \ref{teleig} ist
\begin{eqnarray*}
	T(ir,ir,\ldots) &\stackrel{(3)} {\simeq}& T(r,i,r,i,\ldots) \cr
	                &\stackrel{(2)}{\simeq} & T(i,r,i,r,\ldots) \cr
	                &\stackrel{(3)}{\simeq} & T(ri,ri,\ldots).
	\end{eqnarray*}
Mit Eigenschaft $(1)$ und der Tatsache, dass $ri\simeq id_Y$ ist, folgt, dass $T(ri,ri,\ldots)$ homotopieäquivalent zum Teleskop der Identitätsabbildungen $Y\rightarrow Y\rightarrow Y\rightarrow Y \ldots$ ist, also homotopieäquivalent zu $Y\times[0,\infty)\simeq Y$. Andererseits ist aber $ir$ wegen zellulärer Approximation homotop zu einer zellulären Abbildung zwischen CW-Komplexen $f\colon X\rightarrow X$, woraus erneut mit $(1)$ folgt, dass $T(ir,ir,\ldots)\simeq T(f,f,\ldots)$ ist. Wegen \ref{telcw} ist das ein CW-Komplex.
Insgesamt folgt also, dass $Y$ homotopieäquivalent zu einem CW-Komplex ist.
\end{bew}

Wie für CW-Komplexe üblich, wird im Beweis zu folgendem Satz ein Induktionsargument ausgenutzt. Entscheidend für den Beweis ist weiterhin die Approximation stetiger Abbildungen durch simpliziale Abbildungen sowie der Fakt, dass durch homotope Anklebungsabbildungen entstandene Räume homotopieäquivalent sind. Eine Schwierigkeit stellt hier jedoch dar, dass Gebilde, die ähnlich nützliche Eigenschaften haben wie der Abbildungszylinder, in der simplizialen Kategorie weniger einfach zu definieren sind.\\ Ebensolche werden hier aber zur Konstruktion des Simplizialkomplexes benötigt. Ein vollständiger Beweis findet sich in \cite{MR1867354}(Thm 2C.5.).

\begin{satz}\label{simpCW}
Jeder CW-Komplex $X$ ist homotopieäquivalent zu einem Simplizialkomplex, der so gewählt werden kann, dass er dieselbe Dimension wie $X$ hat, endlich ist, falls $X$ endlich ist und abzählbar ist, falls $X$ abzählbar ist.
\end{satz}

\subsection{Nachweis der Kofinalität}

Es ist nun soweit, dass wir Satz \ref{2.4} beweisen können. Die Ideen des Beweises entstammen \cite{MR1335915}. \\
Zur Erinnerung hier noch einmal die Formulierung:\\
Sei $X$ eine kompakte Teilmenge eines kompakten euklidischen Umgebungsretraktes $\mathrm{E}$. 
Dann gibt es einen Funktor $\vartheta$ zwischen $\Omega_{X}$ und der Kategorie $\Lambda_{X}$, wobei $\Lambda_{X}^{op}=\mathscr{P}^{X}$ ist, definiert durch folgende Vorschrift:\[
\begin{array}{lrll}
\vartheta\colon & \Omega_{X}  &\rightarrow & \Lambda_{X}\cr 
&U & \mapsto &  \left[X\hookrightarrow U\right] \cr
&\left(U<V\right) & \mapsto & [V\hookrightarrow U],
\end{array}
\] wobei in den Homotopieklassenklammern jeweils die Inklusionsabbildungen stehen.\\
$\vartheta$ ist stark kofinal.

\begin{bew}[von Satz \ref{2.4}] 
Als erstes ist zu zeigen, dass $\vartheta$ einen Funktor zwischen den angegebenen Kategorien darstellt. Die Funktoreigenschaften sind klar, es reicht also zu zeigen, dass $\vartheta$ tatsächlich in die Kategorie $\Lambda_{X}$ abbildet.  Damit das gilt, muss jede Umgebung $U$ von $X$ im euklidischen Umgebungsretrakt $\mathrm{E}$ homotopieäquivalent zu einem Polyeder sein.\\
Hierfür ist Satz \ref{off-rn} entscheidend, der besagt, dass offene Teilmengen von Polyedern im $\mathbb R^n$ wieder Polyeder sind.
Insbesondere sind dann alle offenen Teilmengen des $\mathbb R^n$ Polyeder.\\
Sei $U$ eine offene Umgebung von $X$ in $\mathrm{E}$. Da $\mathrm E$ ein euklidisches Umgebungsretrakt ist, ist es per Definition homöomorph zu einem Umgebungsretrakt $Y\subset \mathbb R^n$.\\ Es bezeichne $\varphi\colon E\stackrel{\approx}{\rightarrow} Y$ den Homöomorphismus zwischen $E$ und $Y\subset \mathbb R^n$. Wiederum per Definition existiert eine in $\mathbb R^n$ offene Umgebung $P$ von $Y$ (welche nach dem zitierten Satz ein Polyeder ist) und eine Retraktion $r\colon P\rightarrow Y$.\\ Letzteres bedeutet, dass $Y\stackrel{i}{\hookrightarrow}P\stackrel{r}{\rightarrow}Y$ die Identitätsabbildung ist.
\noindent Definiere $\widetilde{U}:=\varphi(U)\approx U$, welches eine offene Menge in $Y$ darstellt. Daher ist $r^{-1}(\widetilde{U})$ offen im Polyeder $P\subset \mathbb R^n$ und damit selbst ein Polyeder. Betrachte \[
\iota\colon\widetilde{U}\hookrightarrow r^{-1}(\widetilde{U}) \text{ die Inklusionsabbildung und } \tilde{r}:=r_{|_{r^{-1}(\widetilde{U})}}\colon r^{-1}(\widetilde{U})\rightarrow \widetilde{U},
\] wobei erstere Abbildung Sinn macht, da $\widetilde{U}\subset Y$ und damit $\widetilde{U}\subset r^{-1}(\widetilde{U})$  wegen der Retraktionseigenschaft von $r$ gilt. Es ist $\tilde{r}\iota = id_{\widetilde{U}}$, da $r$ eine Retraktion ist, und es folgt, dass $\widetilde{U}$ das Retrakt eines Polyeders ist.\\
Insbesondere ist $\widetilde{U}$  daher dominiert von einem Polyeder. Mit den Sätzen \ref{domCW} und \ref{simpCW} ist $\widetilde{U}$ (und damit $U$) also homotopieäquivalent zu einem Polyeder.

Als nächstes müssen wir die schwache Kofinalität von $\theta$ zeigen:\\
Sei dazu $[\xi\colon X \rightarrow P_\xi]$ aus $\Lambda_{X}$ gegeben. Für schwache Kofinalität ist die Existenz einer offenen Umgebung $U \in \Omega_{X}$ und einer Abbildung $\eta\colon U\rightarrow P_\xi$ nachzuweisen, so dass $\xi\simeq\eta\iota$ gilt, wobei $\iota\colon X\hookrightarrow U$ die Inklusion ist.\\ Bis auf Homotopie muss die Abbildung $\eta$ also auf $X$ mit $\xi$ übereinstimmen. 
Da $X$ eine kompakte Teilmenge des euklidischen Umgebungsretraktes $\mathrm{E}$ ist, findet man eine offene Umgebung $U$ von $X$, in der $X$ abgeschlossen ist. Nach einer verallgemeinerten Form des Tietzeschen Erweiterungssatzes\footnote{Zum Beispiel zu finden in \cite{MR0048821} (Thm 5.5).} ist nun $\xi\colon X\rightarrow P_\xi$ auf ganz $U$ erweiterbar. Diese Erweiterung von $\xi$ liefert eine Abbildung ${\eta}\colon U\rightarrow P_\xi$, die auf $X$ sogar mit der Abbildung $\xi$ übereinstimmt, also insbesondere die gestellte Forderung erfüllt. 

Für die starke Kofinalität seien nun zwei solcher Abbildungen \[\eta_1\colon U_1\rightarrow P_\xi\leftarrow U_2\colon\eta_2\] von offenen Umgebungen $U_i$ von $X$ gegeben, welche auf $X$ (bis auf Homotopie) übereinstimmen. Wir suchen nun eine offene Umgebung $U$ von $X$, für die 
$U\subset U_i$ gilt und auf der $\eta_1$ und $\eta_2$ homotop sind. Das folgende Diagramm soll also bis auf Homotopie kommutieren.\[
\begin{xy}
\xymatrix{
 & U_1 \ar[rd]^{\eta_1} & \cr
U \ar[ru]^{\subset} \ar[rd]_{\subset} & &P_\xi \cr
 & U_2 \ar[ru]_{\eta_2} & \cr
}
\end{xy}
\]

Sei $h_t\colon X \rightarrow P_\xi$ mit $0\leq t\leq1$ eine Homotopie zwischen $\eta_1|_{X}$ und $\eta_2|_{X}$.\\ Aus der Kompaktheit von $X$ folgt wie oben die Existenz einer offenen Umgebung $U$ in der offenen Menge $U_1\cap U_2$, so dass $X$ abgeschlossen in $U$ ist.\\
Definiere die abgeschlossene Teilmenge \[
C:=\left(X\times I\right)\cup\left( U \times\left\{0 \right\}\right)\cup\left( U\times\left\{1\right\}\right)\subset U\times I\] und die stetige Abbildung
\[
d\colon C\rightarrow P_\xi \text{ durch } d(x,t)=\left\{\begin{array}{ll}
																										h_t(x) & \text{ für } x\in X, t\in I \cr
																										\eta_1(x) & \text{ für } x\in U\setminus X, t=0 \cr
																										\eta_2(x)  & \text{ für } x\in U\setminus X, t=1.

\end{array}\right.
\]
Diese kann mit Hilfe eines verallgemeinerten Tietzeschen Erweiterungssatzes\footnote{Wieder zu finden in \cite{MR0048821} (Thm 5.5).} zu einer Abbildung $D\colon U\times I\rightarrow P_\xi$ erweitert werden. Diese ist die gewünschte Homotopie zwischen $\eta_1|_{U}$ und $\eta_2|_{U}$.\\ $\vartheta$ ist also ein stark kofinaler Funktor.

\end{bew}

%% file: dach.tex
Die Dachabbildung auf Kettenebene mit Koeffizienten in $\mathbb Z$ sinnvoll zu definieren, ist das erste Ziel dieses Kapitels. Als Vorbereitung hierauf beschäftigt sich der erste Abschnitt mit dem simplizialen Kreuzprodukt und auch insbesondere mit dessen Beziehung zur koordinatenvertauschenden Involution.\\
Resultierend ergibt sich im zweiten Abschnitt die Definition der Dachabbildung auf Kettenebene in geraden Dimensionen, deren Existenz in \v{C}ech-Homologie (in geraden Dimensionen) im Anschluss daran bewiesen wird.\\
Schließlich soll im letzten Abschnitt beleuchtet werden, welche Eigenschaften die Dachabbildung interessant machen.

Sei von nun an $X$ stets eine kompakte Teilmenge einer glatten, kompakten Mannigfaltigkeit mit Rand  und $(X,A)$ ein kompaktes Paar. 
\[
\tau:X\times X\rightarrow X\times X
\]
\[
(x,y)\mapsto(y,x)
\]
sei die koordinatenvertauschende Involution und $\widetriangle{X}$ sei definiert als \[ \widetriangle{X}:=X\times X /\tau=p_\tau^X(X\times X).\] Für das kompakte Paar $(X,A)$ ist dementsprechend  $\widetriangle{(X,A)}:=(\widetriangle{X},\widetriangle{A})$, wobei \[\widetriangle{A}:=p_\tau^X(X\times A)\cup p_\tau^X(A\times X)\cup p_\tau^X(\bigtriangleup_X)\] ist und $\bigtriangleup_X$ die Diagonale $\left\{(x,x) \mid x\in X\right\}\subset X \times X$ bezeichnet.

\section{Das simpliziale Kreuzprodukt}

Im Folgenden werden wir das simpliziale Kreuzprodukt benötigen, welches wir uns deshalb nun genauer ansehen wollen. Die hier benutzte Definition ist \cite{MR1867354} entnommen.\\
Seien also $Y$ und $Z$ topologische Räume und $\sigma\colon\Delta_m\rightarrow Y$ und $\mu\colon \Delta_n\rightarrow Z$ singuläre Simplizes. Bezeichne die Ecken von $\Delta_m$ mit $ v_0,v_1,\ldots, v_m$ und die von $\Delta_n$ mit $ w_0,w_1,\ldots,w_n$. Weiterhin sei $(\sigma,\mu)\colon \Delta_m\times\Delta_n\rightarrow X\times Y$ die Produktabbildung $(v,w)\mapsto(\sigma(v),\mu(w))$.\\
Die Idee des Kreuzproduktes ist es, $\Delta_m\times\Delta_n$ in $(m+n)$-dimensionale Simplizes zu unterteilen, um dann die Summe über die Einschränkungen der Abbildung $(\sigma,\mu)$ auf diese Simplizes zu bilden, und zwar mit geeigneten Vorzeichen. Hierfür betrachte man $(i,j)$ mit $0\leq i\leq m$, $0\leq j \leq n$ als Ecken eines $(m\times n)-$Gitters im $\mathbb R^2$.\\
Sei $f$ ein Weg vom Punkt $(0,0)$ zum Punkt $(m,n)$ des Gitters, bestehend aus einer Folge von $m+n$ waagerechten und senkrechten Kanten, die immer nur entweder nach oben oder nach rechts führen. Definiere $|f|$ als die Anzahl aller Gitterkästchen unterhalb des Weges $f$. Assoziiere zu $f$ außerdem die lineare Abbildung $l_f\colon\Delta_{m+n}\rightarrow\Delta_m\times\Delta_n$, die der $p$-ten Ecke von $\Delta_{m+n}$ das Eckenpaar $(v_{i_p},w_{j_p})\in\Delta_m\times\Delta_n$ zuordnet, wobei $(i_p,j_p)$ die $p$-te Ecke des Weges $f$ ist. Diese Vorbereitungen münden in der 

\begin{defi}[Kreuzprodukt]
Definiere das simpliziale Kreuzprodukt\[
\times\colon C_m(Y,\mathbb Z)\otimes C_n(Z,\mathbb Z)\rightarrow C_{m+n}(Y\times Z,\mathbb Z)
\] durch die Formel\[ 
\sigma\times\mu:=\sum_{f}{(-1)^{|f|}(\sigma,\mu) \circ l_f}.
\]
\end{defi} 

Um besser zu verstehen, wie die Unterteilung von $\Delta_m\times\Delta_n$ in $(m+n)$-dimensionale Simplizes funktioniert, betrachte man $\Delta_m$ als Teilmenge von $\mathbb R^m$ definiert durch die Koordinaten \begin{equation}
0 \leq x_1 \leq\ldots\leq x_m\leq 1 \text{  mit Ecken } v_i=(\underbrace{0,\ldots,0}_{m-i \text{ Stück}},\underbrace{1,\ldots,1}_{i \text{ Stück}}) \label{delta_m} .\end{equation} Ebenso fasse man $\Delta_n$ als Teilmenge von $\mathbb R^n$ gegeben durch \begin{equation}
0 \leq y_1 \leq\ldots\leq y_n\leq 1 \text{  mit Ecken } w_j=(\underbrace{0,\ldots,0}_{n-j \text{ Stück}},\underbrace{1,\ldots,1}_{j \text{ Stück}}) \label{delta_n} \end{equation} auf.\\
$\Delta_m\times\Delta_n$ sind dann die Tupel $(x_1,x_2,\ldots, x_m,y_1,y_2,\ldots,y_n)$, für die sowohl (\ref{delta_m}) als auch (\ref{delta_n}) gilt. Hierin definiert $0\leq x_1\leq\ldots \leq x_m\leq y_1\leq \ldots\leq y_n\leq 1$ einen $(m+n)$-dimensionalen Simplex.\\ 
Für jeden anderen Punkt $p$ in $\Delta_m\times\Delta_n$ gibt es $0\leq x_1\leq\ldots \leq x_m\leq y_1\leq \ldots\leq y_n\leq 1,$ die durch eine Permutation der Koordinaten in $p$ überführt werden können. So eine Permutation entspricht gerade einem Weg $f$ im $m\times n-$Gitter von oben mit je einer nach rechts führenden Kante für jedes $x_i$ und einer nach oben führenden Kante für jedes $y_i$ in der permutierten Darstellung.\\
So kann $\Delta_m\times\Delta_n$ als Vereinigung der Simplizes $l_f(\Delta_{m+n})$ indiziert durch die Wege $f$ dargestellt werden.

\begin{bem}
Durch Nachrechnen folgt : Für $\sigma\colon\Delta_m\rightarrow Y$ und $\mu\colon\Delta_n\rightarrow Z$ gilt die übliche Randformel\[\partial(\sigma\times\mu)=\partial\sigma\times\mu + (-1)^m\sigma\times\partial\mu.\]
\end{bem}

Folgende Eigenschaft des simplizialen Kreuzprodukts ist von Bedeutung für das weitere Vorgehen:

\begin{lem} \label{eigkreuz}

 Bezeichne $p_{\tau}^X\colon X\times X\rightarrow \widetriangle{X}$ die Projektion und $(p_\tau^X)_\sharp$ die davon induzierte Kettenabbildung, so 				gilt für singuläre Simplizes $\sigma_1,\sigma_2\colon\Delta_k\rightarrow X$: 
							\begin{enumerate}
									\item $(p_\tau^X)\sharp(\sigma_1\times\sigma_2)=(p_\tau^X)\sharp(\sigma_2\times\sigma_1)$, falls k gerade.
									\item $(p_\tau^X)\sharp(\sigma_1\times\sigma_2)= -(p_\tau^X)\sharp(\sigma_2\times\sigma_1)$ , falls k ungerade.
									\end{enumerate}

\end{lem}

\begin{bew}
 	            Zu jedem Weg $f$ im $k\times k$-Gitter ist der an der Diagonalen gespiegelte Weg $\bar{f}$ ungleich $f$ und es gilt:
							\[\tau\left(\left(\sigma_2,\sigma_1\right)\circ l_f \right)=\left(\sigma_1,\sigma_2\right)\circ l_{\bar{f}}.
							\] Betrachte  
							\begin{eqnarray*}
							\tau(\sigma_2\times\sigma_1) &=& \tau \left( \sum_f{(-1)^{|f|}(\sigma_2,\sigma_1)\circ l_f} \right)\\
																					 &=& \sum_f{(-1)^{|f|}\tau(\sigma_2,\sigma_1)\circ l_f}	\\
																					 &=& \sum_f{(-1)^{|f|}(\sigma_1,\sigma_2)\circ l_{\bar{f}}}.\\
							\end{eqnarray*}
							Da $|\bar{f}|+|f|=k^2$ gelten muss, folgt:\\ Ist $k$ gerade, so ist $|f|$ genau dann gerade (respektive ungerade), wenn es $|\bar{f}|$ ist.\\ Ist 								$k$ ungerade, so  ist $|f|$ genau dann gerade (beziehungsweise ungerade), wenn $|\bar{f}|$ ungerade (beziehungsweise gerade) ist.  Damit ist 
							\begin{eqnarray*}
							\tau(\sigma_2\times\sigma_1)&=& \sum_f{(-1)^{|\bar{f}|}(\sigma_1,\sigma_2)\circ l_{\bar{f}}}\cr
																					&=& \sigma_1\times \sigma_2 \text{\hspace{3,85em}, falls $k$ gerade	und}
							\end{eqnarray*}							
							\begin{eqnarray*}																												
							\tau(\sigma_2\times\sigma_1)&=& (-1)\sum_f{(-1)^{|\bar{f}|}(\sigma_1,\sigma_2)\circ l_{\bar{f}}} \cr 
																					&=& (-1)(\sigma_1\times\sigma_2) \text{\hspace{1,3em}  , falls $k$ ungerade.}     
							\end{eqnarray*}Daraus folgt die Behauptung.
		
\end{bew}

\section{Definition und Existenz in der \v{C}ech-Homologie}

Obiges Lemma zeigt, dass die nun folgende Definition nur für gerade Dimensionen sinnvoll formuliert werden kann.
\begin{defi}[Dachabbildung für singuläre Ketten] 
Sei $k\in\mathbb Z$ eine gerade Zahl und sei $\, C_k^{sing}(X,A;\mathbb Z)$ die $k$-te singuläre Kettengruppe von $(X,A)$. Definiere 

\[ 
\widehat{\hspace{0,5em}\cdot\hspace{0,5em}}\, \colon \, C_k^{sing}(X,A;\mathbb Z)\rightarrow C_{2k}^{sing}(\widetriangle{(X,A)};\mathbb Z)
\]
durch

\[
\sigma=\sum_{i=1}^n {g_i\sigma_{i}} \mapsto \widehat{\sigma} :=
 \sum\limits_{\genfrac{}{}{0pt}{}{i<j}{1\leq i,j \leq n}}{g_i g_j(p_{\tau}^X)_{\sharp}(\sigma_i \times \sigma_j)},
\]wobei $\times$ das simpliziale Kreuzprodukt und $(p_{\tau}^X)_{\sharp}$ die von der Projektion $p_{\tau}^X:X\times X \rightarrow (X \times X)/\tau$ induzierte Kettenabbildung ist. 

\end{defi}

Die Dachabbildung ist wohldefiniert, denn zum einen setzt sie sich aus wohldefinierten Abbildungen zusammen, zum anderen ist sie von der Summationsreihenfolge unabhängig. Denn sei $\pi$ eine Permutation der Menge $\{1,\ldots,n\}$, dann gilt:
\[
\sum\limits_{\genfrac{}{}{0pt}{}{i<j}{1\leq i,j \leq n}}{g_i g_j (p_{\tau}^X)_{\sharp}(\sigma_i \times \sigma_j)}=\sum\limits_{\genfrac{}{}{0pt}{}{i<j}{1\leq i,j \leq n}}{g_{\pi(i)}g_{\pi(j)}(p_{\tau}^X)_{\sharp}(\sigma_{\pi(i)}\times\sigma_{\pi(j)})},
\]
was aus der nur für gerade $k$ gültigen Gleichung $(p_{\tau}^X)_{\sharp}(\sigma_i\times\sigma_j)=(p_{\tau}^X)_{\sharp}(\sigma_j\times\sigma_i)$ folgt. 

Um zu beweisen, dass die Definition der Dachabbildung auf Kettenebene eine Abbildung in \v{C}ech-Homologie induziert, ist es erforderlich, relative Homologie zu betrachten. Folgende Definition beziehungsweise der Fakt, dass diese eine injektive Abbildung liefert, werden später sehr nützlich sein.

\bigskip

\begin{defi} \label{injAbb}
Sei $X$ ein topologischer Raum, $V \subset X$ und offene Mengen $ \{U_{\lambda}\}_{\lambda \in \Lambda} \subset X $ gegeben, so dass \begin{itemize}
	\item  $V\cap U_{\lambda}=\emptyset \text{, also } V\subset X\setminus U_{\lambda}$ für alle $\lambda \in \Lambda$ gilt und
	\item  die Mengen $ \{U_{\lambda}\}_{\lambda \in \Lambda}$ $X\setminus V$ überdecken. 
\end{itemize}
Ausgehend von der Familie kurzer exakter Tripel-Sequenzen (des Tripels $(X,X\setminus U_\lambda, V))$ von Kettenkomplexen\begin{equation}
\left\{ 0 \rightarrow C_{\ast}^{sing}(X\setminus U_{\lambda},V;\mathbb Z)\stackrel {i^{\lambda}_{\sharp}} \longrightarrow C_{\ast}^{sing}(X,V;\mathbb Z)\stackrel {\pi^{\lambda}} \longrightarrow C_{\ast}^{sing}(X,X\setminus U_{\lambda};\mathbb Z)\rightarrow 0 \right\}_{\lambda\in\Lambda}
\label{exakt}
\end{equation}
definieren wir die Abbildung \[
\alpha \colon C_{\ast}^{sing}(X,V;\mathbb Z)\longrightarrow \prod_{\lambda \in \Lambda}{C_{\ast}^{sing}(X,X\setminus U_{\lambda};\mathbb Z)}
\]
durch die Vorschrift\[
\sigma \longmapsto \left( \alpha_{\sigma} \colon \Lambda \rightarrow \bigcup_{\lambda \in \Lambda}C_{\ast}^{sing}(X,X\setminus U_{\lambda};\mathbb Z), \lambda \mapsto \pi^{\lambda}(\sigma)   \right).
\]
\end{defi}

\begin{lem}[Injektivität] Die eben definierte Abbildung ist injektiv.
\end{lem}
\begin{proof}
Sei $\sigma\in ker(\alpha)$, also $\sigma\in C^{sing}_{\ast}(X,V;\mathbb Z)$ mit $\alpha(\sigma)=\alpha_{\sigma}\equiv 0$, das heißt \[ \pi^{\lambda}(\sigma)=0\in C_{\ast}^{sing}(X,X\setminus U_{\lambda};\mathbb Z) \text{ für alle } \lambda\in\Lambda.\]
Die Exaktheit der Sequenz (\ref{exakt}) liefert für alle $\lambda\in\Lambda$ ein Element $\sigma_{\lambda}\in C^{sing}_{\ast}(X\setminus U_{\lambda},V;\mathbb Z)$ mit \begin{equation}
i^{\lambda}_{\sharp}(\sigma_{\lambda})=i^{\lambda}\circ\sigma_{\lambda}=\sigma. 
\label{inklusion}
\end{equation}
Da die Abbildung $i_{\sharp}^{\lambda}$ durch die Inklusion $i^{\lambda}\colon X\setminus U_{\lambda}\hookrightarrow X$ induziert wird, bedeutet Gleichung (\ref{inklusion}), dass $\sigma$ bereits ein Element von $C_{\ast}^{sing}(X\setminus U_{\lambda},V;\mathbb Z)$ ist, und zwar für alle $\lambda \in \Lambda$. Wegen der Überdeckungseigenschaft der Mengen $\{U_{\lambda}\}_{\lambda \in \Lambda}$ heißt das aber, dass $\sigma=0$ in $C_{\ast}^{sing}(X,V;\mathbb Z)$ ist. Damit ist gezeigt, dass $ker(\alpha)=0$, also $\alpha$ injektiv ist.
\end{proof}

Schließlich ist nun zu zeigen, dass die Dachabbildung auch in \v{C}ech-Homologie existiert:

\begin{satz} Sei $k \in\mathbb Z$ eine gerade Zahl. Vereinbarungsgemäß sei weiter $(X,A)$ ein kompaktes Paar einer glatten, kompakten $\partial-$Mannigfaltigkeit.  Dann induziert die Dachabbildung  $\widehat{\cdot}\, \colon \, C_k^{sing}(X,A;\mathbb Z)\rightarrow C_{2k}^{sing}(\widetriangle{(X,A)};\mathbb Z)$ für singuläre Ketten eine Dachabbildung \[ 
\widehat{\cdot} \,\colon \check{H}_{k}(X,A;\mathbb Z)\rightarrow \check{H}_{2k}\left(\widetriangle{(X,A)};\mathbb Z \right)
\]
in \v{C}ech-Homologie.
\end{satz}

\begin{proof}
Es ist hier $ \check{H}_{k}(X,A;\mathbb Z)\cong H_{k}^{sing}(X,A;\mathbb Z).$\footnote{Siehe beispielsweise \cite{MR1335915}(Prop. 13.17).} Weiterhin ist die \v{C}ech-Homologie isomorph zum inversen Limes über die relative singuläre Homologie offener Umgebungen, wie im ersten Kapitel gezeigt wurde: 
\[
\check{H}_{2k}\left(\widetriangle{(X,A)};\mathbb Z \right) \cong \varprojlim \left\{H_{2k}^{sing}(\widetriangle{X}, V_{\widetriangle{A}};\mathbb Z) \mid  V_{\widetriangle{A}}   \supset \widetriangle{A} \text{  offen} \right\}.
\]
Daher ist für $\sigma \in H_k^{sing}(X,A;\mathbb Z)$, also $\sigma=\sum_{i=1}^n{g_i\sigma_i} \in C_k^{sing}(X,A;\mathbb Z)$ mit $\partial\sigma=0$, zu betrachten, was mit $\hat{\sigma}$ in $H_{2k}^{sing}(\widetriangle{X},V_{\widetriangle{A}};\mathbb Z) $ passiert für ein offenes $V_{\widetriangle{A}} \supset \widetriangle{A}$.\\
Wähle eine Menge $W_{\widetriangle{A}}\subset V_{\widetriangle{A}}$ so aus, dass  $\bigtriangleup_X \subset (p_\tau^X)^{-1}(\widetriangle{A})\subset(p_\tau^X)^{-1}(W_{\widetriangle{A}})$ und $\overline{W_{\widetriangle{A}}}\subset V_{\widetriangle{A}}$ gilt. Setze von nun an $ V := (p_\tau^X)^{-1}(W_{\widetriangle{A}})$.\\
Da außerhalb der Diagonalen $p_{\tau}^X|_{X\times X\setminus\bigtriangleup_X}\colon X\times X\setminus\bigtriangleup_X\rightarrow \widetriangle{X}\setminus p_{\tau}^X(\bigtriangleup_X)$ eine Zweiblättrige Überlagerung ist und insbesondere $\bigtriangleup_X \subset \overline{V}$ gilt, ist es möglich, für jedes  $x\in X\times X\setminus\overline{V}$ eine offene Umgebung $U_x$ disjunkt zu $\overline{V}$ zu finden, so dass $p_{\tau}^X|_{U_x}\colon U_x\stackrel{\approx}{\rightarrow}p_{\tau}^X(U_x)$ ein Homöomorphismus ist. Zu bemerken ist, dass dann $\overline{W_{\widetriangle{A}}}\subset\widetriangle{X}\setminus p_\tau^X(U_x)$ gilt für alle $x\in X\times X\setminus V$.\\
Wir können annehmen, dass $\sigma$ so unterteilt ist, dass entweder
\begin{eqnarray}	 \label{schnitt1} 
im(\sigma_i\times \sigma _j) & \subset & V \notag \\
	 \text{oder} &  & \\
	   im(\sigma_i\times \sigma_j)\cap \bigtriangleup_X & = & \emptyset \notag
\end{eqnarray} 
gilt $\forall i,j \in \{1,\ldots,n\}$.

\noindent Sei nun $x\in X\times X \setminus\overline{V}$ fest gewählt. Betrachte\[
\pi^x\colon C_{2k}^{sing}(\widetriangle{X}, \overline{W_{\widetriangle{A}}} ;\mathbb Z)\rightarrow C_{2k}^{sing}(\widetriangle{X},\widetriangle{X}\setminus p_\tau ^X(U_x);\mathbb Z),
\]welches wir schon zur Definition der Abbildung in Definition \ref{injAbb} benutzt haben. (Durch das Dazwischenschalten einer weiteren Projektion kann $(p_{\tau}^X)_{\sharp}(\sigma\times\sigma)$ als Element von $C_{2k}^{sing}(\widetriangle{X}, \overline{W_{\widetriangle{A}}} ;\mathbb Z)$ angesehen werden.)\\
Es ist
\begin{eqnarray}
\pi^x\left((p_{\tau}^X)_{\sharp}(\sigma\times\sigma)\right)& = & \pi^x\left( \sum_{1\leq i,j\leq n}{g_i g_j(p_{\tau}^X)_{\sharp}(\sigma_i\times \sigma _j)}\right) \notag \\
& \stackrel{(I)} {=} & \pi^x\left(\sum\limits_{\genfrac{}{}{0pt}{}{i\neq j}{1\leq i,j\leq n}}{g_i g_j(p_{\tau}^X)_{\sharp}(\sigma_i\times\sigma_j)}\right) \notag \\
& \stackrel{ (II) } {=} & \pi^x\left( \sum\limits_{\genfrac{}{}{0pt}{}{i<j}{1\leq i,j\leq n}}{g_i g_j(p_{\tau}^X)_{\sharp}(\sigma_i \times\sigma_j)}\right) \notag \\
& = &\pi^x\left(\hat{\sigma}\right)\in C^{sing}_{2k}(\widetriangle{X},\widetriangle{X}\setminus p_{\tau}^X(U_x);\mathbb Z),
\label{lokal}
\end{eqnarray} 

\noindent wobei $(I)$ gilt, da $im(\sigma_i\times\sigma_i)\cap \bigtriangleup_X \neq\emptyset$ für alle $i\in\{1,\ldots,n\}$ ist und deshalb\linebreak[4]
 $im(\sigma_i\times\sigma_i)\subset V \subset X\times X \setminus U_x$ für alle $i$.
Gleichung $(II)$ gilt, denn die Abbildung  \[
B\colon\genfrac{\{}{\}}{0pt}{0}{(i,j) \in \{1,\ldots,n\}^2; i\neq j}{im(\sigma_i\times \sigma_j)\cap U_x \neq \emptyset}\rightarrow \genfrac{\{}{\}}{0pt}{0}{(i,j)\in\{1,\ldots,n\};i<j}{im((p_{\tau}^X)_{\sharp}(\sigma_i\times\sigma_j))\cap(p_{\tau}^X)(U_x)\neq \emptyset}
\]\[(i,j)\mapsto \left\{\begin{array}{ll} (i,j) & \text{ , falls $i<j$}\cr
														 (j,i) & \text{ , sonst} 
														\end{array}\right.
\] ist eine Bijektion. Um das zu zeigen, sei $(i,j)$ mit $i<j$ gegeben mit der Eigenschaft, dass $im(p_{\tau}^X)_{\sharp}(\sigma_i\times\sigma_j)\cap(p_{\tau}^X)(U_x)\neq \emptyset$. Es gibt also $p\in im(p_{\tau}^X)_\sharp(\sigma_i\times\sigma_j)=im(p_{\tau}^X)_\sharp(\sigma_j\times\sigma_i)$ und passend $q=(q_1,q_2)\in U_x$, so dass \[ p=(p_\tau^X)(q)\Leftrightarrow p=[(q_1,q_2),(q_2,q_1)] \in im(p_{\tau}^X)_\sharp(\sigma_i\times\sigma_j)=im(p_{\tau}^X)_\sharp(\sigma_j\times\sigma_i)
\] \[
\Rightarrow q\in U_x \cap im(\sigma_i \times\sigma_j) \text{ oder }q\in U_x\cap im(\sigma_j\times\sigma_i).\] Das bedeutet aber nichts anderes als \[ B(i,j)=(i,j) \text{ oder } B(j,i)=(i,j),\] also die Surjektivität von $B$. Injektivität folgt leicht. Die so bewiesene Gleichung hilft weiter:

\noindent Da $\partial\sigma=0$ ist, ist $\partial(\sigma\times\sigma)=0$ und wegen (\ref{lokal}) dann auch \begin{equation}
0=\partial\pi^x(\hat{\sigma}) \in C_{2k-1}^{sing}(\widetriangle{X},\widetriangle{X}\setminus p_{\tau}^X(U_x);\mathbb Z) \label{rand0}.
\end{equation}

\noindent Die Abbildung aus Definition \ref{injAbb} kann nun nützlich sein: Für\[
\alpha \colon C_{2k-1}^{sing}(\widetriangle{X},\overline{W_{\widetriangle{A}}};\mathbb Z)\rightarrow \prod_{x\in X\times X \setminus\overline{V}} C_{2k-1}^{sing}(\widetriangle{X},\widetriangle{X} \setminus p_{\tau}^X(U_x);\mathbb Z)  
\] 
gilt wegen (\ref{rand0})\[
\alpha(\partial \hat{\sigma})(x)=\alpha_{\partial \hat{\sigma}}(x)=\partial\pi^x(\hat{\sigma})=0\text{ für alle }x\in X\times X\setminus\overline{V}.
\]

\noindent Die Mengen $\{U_x\}_{x\in X\times X\setminus \overline{V}}$ erfüllen die Bedingungen aus Definition \ref{injAbb}, daher ist $\alpha$ injektiv. Es folgt $\partial\hat{\sigma}=0 \in C_{2k-1}^{sing}(\widetriangle{X},\overline{W_{\widetriangle{A}}};\mathbb Z)$. Die Bedingung $\overline{W_{\widetriangle{A}}}\subset V_{\widetriangle{A}}$ liefert dann\[
\partial\hat{\sigma}=0 \in C_{2k-1}^{sing}(\widetriangle{X},V_{\widetriangle{A}};\mathbb Z).
\]

Es bleibt die Wohldefiniertheit der Dachabbildung auf (\v{C}ech-)Homologieebene nachzuweisen. Für diesen Nachweis der Wohldefiniertheit seien $\sigma$ und $\mu$ mit $\left[ \sigma\right]=\left[ \mu\right] \in H_k^{sing}(X,A;\mathbb Z)$ gegeben. Es gibt also $\left[ \nu \right] \in H_{k+1}^{sing}(X,A,\mathbb Z)$, so dass $\partial\nu=\sigma-\mu$ ist. Passend zu Eigenschaft (\ref{schnitt1}) von $\sigma$ und $\mu$ kann für $\nu=\sum_{i=1}^m{k_i\nu_i}$ eine Unterteilung vorausgesetzt werden, so dass gilt:
\begin{equation}
	im((\nu_i\times\nu_i))\subset (p_\tau^X)^{-1}(V_{\widetriangle{A}})  \; \forall i \in \{1,\ldots,m\}.
	\label{schnitt2}
\end{equation}  
Zu zeigen ist, dass $\left[\hat{\sigma}\right]=\left[\hat{\mu}\right]=\left[\widehat{\sigma-\partial\nu}\right]\in H_{2k}^{sing}(\widetriangle{X},V_{\widetriangle{A}};\mathbb Z)$ gilt. Hierfür reicht es zu zeigen, dass \[ 
\left[\hat{\sigma}\right]=\left[\widehat{\sigma-\partial\nu_i}\right] \, \forall i\in \{1,\ldots,m\},  
\] denn die gewünschte Aussage folgt dann induktiv.
Da die Dachabbildung unabhängig von der Summationsreihenfolge ist, können nach eventueller Umnummerierung und Umbenennung $\sigma$ und $\partial\nu_i$ wie folgt dargestellt werden: \begin{eqnarray*}
\sigma & = & \sum_{i=1}^n{g_i\sigma_i}+\sum_{i=n+1}^N{g_i\sigma_i},\\
\partial\nu_i & = & \sum_{i=n+1}^N{g_i\sigma_i}+\sum_{i=N+1}^{N+k}{g_i\sigma_i}.	
\end{eqnarray*}
Hier ist es möglich, dass $n=N$ gilt, was dann bedeutet, dass $\sigma$ und $\partial\nu_i$ keine Simplizes gemeinsam haben. Es gilt:

\[\sigma-\partial\nu_i=\sum_{i=1}^n{g_i\sigma_i}+\sum_{i=N+1}^{N+k}{(-1)g_i\sigma_i}.
\]
Die Dachabbildung angewendet ergibt \begin{eqnarray*}
\hat{\sigma} & = & \sum_{1\leq i < j\leq n}{g_i g_j(p_{\tau}^X)_{\sharp}(\sigma_i\times\sigma_j)} + \sum_{1\leq i\leq n < j\leq N}{g_i g_j(p_{\tau}^X)_{\sharp}(\sigma_i\times\sigma_j)} + \Theta_1 \text{, wobei }\\
\Theta_1 & = & \sum_{n<i<j\leq N}{g_i g_j(p_{\tau}^X)_{\sharp}(\sigma_i\times\sigma_j)} \text{ und }\\
\widehat{\sigma-\partial\nu_i} & = & \sum_{1\leq i<j\leq n}{g_i g_j(p_{\tau}^X)_{\sharp}(\sigma_i\times\sigma_j)} \\ & & +\sum_{1\leq i\leq n\leq N <j \leq N+k}{(-1)g_i g_j(p_{\tau}^X)_{\sharp}(\sigma_i\times\sigma_j)} +\Theta_2 \text{, wobei}\\
\Theta_2 & = & \sum_{N<i<j\leq N+k}{g_i g_j(p_{\tau}^X)_{\sharp}(\sigma_i\times\sigma_j)} \text{ und damit }\\
	\end{eqnarray*}
	
\begin{eqnarray*}
\hat{\sigma} - \widehat{\sigma-\partial\nu_i} & = & \sum_{1\leq i\leq n<j\leq N+k}{g_i g_j(p_{\tau}^X)_{\sharp}(\sigma_i\times\sigma_j)}+\Theta_1-\Theta_2\\
  & = & (p_{\tau}^X)_{\sharp}\left(\sum_{i=1}^n{g_i\sigma_i}\times \sum_{j=n+1}^{N+k}{g_j\sigma_j}\right) +\Theta_1 -\Theta_2 +\Theta_3 -\Theta_3 \text{ , wobei}\\
\Theta_3 & = &  (p_{\tau}^X)_{\sharp}\left(\sum_{i=n+1}^N{g_i\sigma_i}\times \sum_{j=n+1}^{N+k}{g_j\sigma_j}\right).
\end{eqnarray*}

\noindent Daraus folgt \begin{eqnarray}
\hat{\sigma} - \widehat{\sigma-\partial\nu_i} & = & (p_{\tau}^X)_{\sharp}(\sigma\times \partial\nu_i) + \Theta_1 -\Theta_2 -\Theta_3 \notag \\
 & = & (p_{\tau}^X)_{\sharp}(\partial\nu_i\times\sigma) + \Theta_1 -\Theta_2 -\Theta_3 \notag \\ 
 & = & (p_{\tau}^X)_{\sharp}(\partial(\nu_i\times\sigma))+\Theta_1-\Theta_2-\Theta_3 \text{ , da }\partial\sigma=0 \notag \\ 
 & = & \partial((p_{\tau}^X)_{\sharp}(\nu_i\times\sigma))+\Theta_1-\Theta_2-\Theta_3. \label{diff}
 \end{eqnarray}

\noindent $\Theta_1, \Theta_2$ und $\Theta_3$ enthalten jeweils nur Summanden der Form $g_i g_j (p_{\tau}^X)_{\sharp}(\sigma_i\times\sigma_j)$, wobei $\sigma_i$ und $\sigma_j$ beides Simplizes aus $\partial\nu_i$ sind. Eigenschaft (\ref{schnitt2}) liefert somit, dass die Bilder\\ $im((p_{\tau}^X)_{\sharp}(\sigma_i\times\sigma_j))$, die in $\Theta_1,\Theta_2$ und $\Theta_3$ vorkommen, Teilmengen der Umgebung $V_{\widetriangle{A}}$ sind. Es ist also wegen (\ref{diff}) \[ 
\left[\hat{ \sigma}\right]=\left[\widehat{\sigma-\partial\nu_i}\right]\in H_{2k}(\widetriangle{X},V_{\widetriangle{A}};\mathbb Z),
\]
was die gewünschte Aussage liefert.\\
\noindent Insgesamt haben wir nun gezeigt, dass wir für $\sigma \in H_k^{sing}(X,A;\mathbb Z) $ und für eine feste, offene Umgebung $V_{\widetriangle{A}}\supset \widetriangle{A}$ ein wohldefiniertes Element $\hat{\sigma}=:\hat{\sigma}_{V_{\widetriangle{A}}}\in H^{sing}_{2k}(\widetriangle{X},V_{\widetriangle{A}},\mathbb Z)$ bekommen. Diese Wohldefiniertheit soll aber auch in der \v{C}ech-Homologie gelten.
Bezeichne mit $\Omega_{\widetriangle{A}}$ die Menge aller offenen Umgebungen von $\widetriangle{A}$ in $\widetriangle{X}$, gerichtet durch umgekehrte Inklusion. Satz \ref{invlim} auf Seite \pageref{invlim} liefert: \begin{eqnarray*}
\check{H}_k(\widetriangle{(X,A)}),\mathbb Z) & \cong & \varprojlim\{H_k^{sing}(\widetriangle{X},V_{\widetriangle{A}};\mathbb Z) |                                                                                     V_{\widetriangle{A}}\supset \widetriangle{A}\} \cr
	                                           & \cong & \{(\sigma_{V_{\widetriangle{A}}})_{V_{\widetriangle{A}}\in \Omega_{\widetriangle{A}}}\in                                                            \prod\limits_{V_{\widetriangle{A}}\in\Omega_{\widetriangle{A}}}                                                                                      H_k^{sing}(\widetriangle{X},V_{\widetriangle{A}}, \mathbb Z) |
                                                                                   (i^{V_{\widetriangle{A}}}_{W_{\widetriangle{A}}})_{\ast}\sigma_{V_{\widetriangle{A}}}                                                               =\sigma_{W_{\widetriangle{A}}} \text{ }\forall \text{ }                                     
	          	                                           V_{\widetriangle{A}} \subset W_{\widetriangle{A}} \} , \cr
\end{eqnarray*} wobei $(i^{V_{\widetriangle{A}}}_{W_{\widetriangle{A}}})_{\ast}$ die von der Inklusion $i^{V_{\widetriangle{A}}}_{W_{\widetriangle{A}}}\colon V_{\widetriangle{A}}\hookrightarrow W_{\widetriangle{A}} $ induzierte Abbildung ist.
Bilde das Element\[
\check{\hat{\sigma}}:=(\hat{\sigma}_{V_{\widetriangle{A}}})_{V_{\widetriangle{A}}\in \Omega_{\widetriangle{A}}}\in \prod\limits_{V_{\widetriangle{A}}\in               \Omega_ {\widetriangle{A}}} H_k^{sing}(\widetriangle{X},V_{\widetriangle{A}}, \mathbb Z),
\]welches auf Grund der Wohldefiniertheit aller seiner Komponenten selbst wohldefiniert ist. Weiterhin gilt aber auch \[ (i^{V_{\widetriangle{A}}}_{W_{\widetriangle{A}}})_{\ast}\hat{\sigma}_{V_{\widetriangle{A}}}=\hat{\sigma}_{W_{\widetriangle{A}}} \text{ }\forall \text{ } V_{\widetriangle{A}} \subset W_{\widetriangle{A}}.
\]
Das bedeutet aber nichts anderes, als dass $\check{\hat{\sigma}}$ wohldefiniert in $\check{H}_k(\widetriangle{(X,A)}),\mathbb Z)$ ist.

\end{proof}

%% file: orient.tex
Im Folgenden soll erläutert werden, warum es sinnvoll ist, die Dachabbildung konstruiert zu haben. Hierfür ist ein wenig Vorarbeit nötig. Die Zusammenfassung der Ergebnisse, die "`Giebelmannigfaltigkeiten"' und ihre Fundamentalklassen betreffen, beruft sich auf \cite{Hage2006}. Dort wurde allerdings mit $\mathbb Z_2$-Koeffizienten gearbeitet, also Orientierung vernachlässigt. Aus diesem Grund wird hier zunächst die Orientierung in unserem speziellen Zusammenhang näher beleuchtet.

\subsection{Orientierung}

\begin{satz}
Sei $M$ eine orientierte Mannigfaltigkeit der Dimension $dim(M)=k$. Die koordinatenvertauschende Involution $\tau:M\times M\rightarrow M\times M$, $(x,y)\mapsto(y,x)$ ist orientierungserhaltend für gerade Dimensionen $k$ und orientierungsumkehrend für ungerade Dimensionen $k$.
\end{satz}

\begin{proof}
$M$ ist orientiert, das heißt, dass es für alle $x\in M$ eine lokale Orientierung $\mu_x\in H_k(M,M\setminus x)$ gibt, für die eine Umgebung $U\ni x$ und ein Element $\mu_U\in H_k(M,M\setminus U)$ existieren, so dass für alle $y\in U$ gilt: \[
\mu_U\mapsto \mu_y
\] durch die Abbildung\[
H_k(M,M\setminus U)\rightarrow H_k(M,M\setminus y).\]
$M\times M$ ist dann auch orientiert:

\noindent Sei $(p,q)\in M\times M$. 
Wähle für $x=(p,q)\in M\times M$ die lokale Orientierung $\mu_x=\mu_p\times\mu_q$, wobei $\mu_p$ die lokale Orientierung an $p$, $\mu_q$ die lokale Orientierung an q in $M$ und $"\times"$ das vom simplizialen Kreuzprodukt kommende homologische Kreuzprodukt ist.\\
Es ist $(p,q)\in U_p\times U_q=:U_{(p,q)}$ und $\mu_{U_{(p,q)}}:=\mu_{U_p}\times \mu_{U_q},$ wobei $M\supset U_p\ni p$ und $M\supset U_q \ni q$ so gewählt sind, dass $\mu_{U_{p}}\mapsto\mu_{\bar{p}}$ durch die Abbildung $H_k(M,M\setminus U_p)\rightarrow H_k(M,M\setminus \bar{p})$ und $\mu_{U_{q}}\mapsto\mu_{\bar{q}}$ durch die Abbildung $H_k(M,M\setminus U_q)\rightarrow H_k(M,M\setminus \bar{q})$ für alle $\bar{p}\in U_p, \bar{q}\in U_q$. Betrachte das folgende kommutative Diagramm: \[
\begin{xy}
\xymatrix{H_{2k}(M\times M,M\times M\setminus U_{(p,q)}) \ar[d]^{\cong} \ar[r]^{\alpha} & H_{2k}(M\times M,M\times M\setminus(\bar{p},\bar{q}))\ar[d]^{\cong} \\
          H_{2k}(M,M\setminus U_p)\otimes H_{2k}(M,M\setminus U_q \ar[r]) & H_{2k}(M,M\setminus \bar{p})\otimes H_{2k}(M,M\setminus \bar{q})}
\end{xy} \]
Die senkrechten Isomorphismen sind durch die Künneth-Formel begründet. Dieses Diagramm und die obigen Eigenschaften belegen, dass $\alpha(\mu_{U_{(p,q)}})=\mu_{\bar{p}}\times\mu_{\bar{q}}=\mu_{(\bar{p},\bar{q})}$ für alle $(\bar{p},\bar{q})\in U_{(p,q)}$.\\
Also kann $M\times M$ wie oben definiert orientiert werden.
In Lemma \ref{eigkreuz} haben wir gezeigt, dass \[
\begin{array}{llrl}
 \tau(\mu_p\times\mu_q)  &=&  \mu_q\times\mu_p\text{ }    &\text{ , falls k gerade und}\cr
   \tau(\mu_p\times\mu_q)&=& -(\mu_q\times\mu_p) &\text{ , falls k ungerade}\cr
\end{array}\] gilt, was die Behauptung liefert.

\end{proof}

Als lokale Orientierung für $\left[(p,q),(q,p)\right] \in \widetriangle{M}=M\times M / \tau$ kann also in dem Fall, in dem $dim(M)$ gerade ist, $(p_\tau^M)_\ast(\mu_p\times\mu_q)$ gewählt werden. Hierbei ist $(p_\tau^M)_\ast$ die von der Projektion $p\colon M\times M\rightarrow \widetriangle{M}$ induzierte Abbildung. Mit dem obigen Ergebnis ist diese Wahl wohldefiniert und damit wissen wir, dass für gerade-dimensionale, orientierte Mannigfaltigkeiten M auch $\widetriangle M$ orientierbar ist.\\
Bis zu diesem Zeitpunkt wissen wir allerdings nicht, ob für eine Mannigfaltigkeit $M$ auch $M\times M$ oder $\widetriangle{M}$ eine Mannigfaltigkeitsstruktur aufweisen. Diese Probleme sollen im Folgenden erläutert werden.

\subsection{Giebelmannigfaltigkeit und Fundamentalklasse}

Für randlose, glatte Mannigfaltigkeiten $M$, das heißt glatte Mannigfaltigkeiten mit\linebreak[4] $\partial M=\emptyset$, kann ohne großen Aufwand gezeigt werden, dass $M\times M$ die Mannigfaltigkeits\-struktur von $M$ erbt.\\
Im Fall von berandeten, glatten Mannigfaltigkeiten $M$ ist es mit größerem Aufwand verbunden zu zeigen, dass auch $(M,\partial M)\times(M,\partial M)$ eine berandete, glatte Mannigfaltigkeits\-struktur aufweist. Es wird die Begriffsbildung eines Kragens einer Mannigfaltigkeit und das Verfahren der Winkelglättung benötigt. Hierfür sei auf \cite{MR0358848}(Abschnitt 13) verwiesen.\\
Die koordinatenvertauschende Involution $\tau\colon M\times M\rightarrow M\times M$ ist nicht fixpunktfrei, deswegen trägt $M\times M/\tau=\widetriangle{M}$ auf den ersten Blick nicht die Struktur einer Mannigfaltigkeit. Indem die Fixpunktmenge, also die Diagonale $ \bigtriangleup_{M}$, "`herausgeschnitten wird"', kann dieses Problem gelöst werden. Um genauer zu verstehen, inwiefern dann "`das verminderte $\widetriangle{M}$"' eine Mannigfaltigkeitsstruktur aufweist, sei hier zunächst an die Definition des transversalen Schnittes zweier Untermannigfaltigkeiten erinnert.\\
Dieser Abschnitt soll vor allem eine Zusammenfassung der benötigten Ergebnisse liefern. Details und Beweise können in \cite{Hage2006} nachgelesen werden.

\begin{defi}[Transversaler Schnitt von Untermannigfaltigkeiten]
Sei $W$ eine glatte Mannigfaltigkeit. Die eingebetteten Untermannigfaltigkeiten $M_1$ und $M_2$ schneiden sich transversal, falls für alle Punkte $p\in M_1\cap M_2$ der Tangentialraum $T_pW$ die Summe der Tangentialräume $T_pM_1$ und $T_pM_2$ ist. Das soll hier heißen:\[
T_pW=T_pM_1+T_pM_2:=\left\{v+w | v\in T_pM_1, w\in T_pM_2\right\}.\] Man notiert $M_1\pitchfork M_2.$
\end{defi}

Die nachfolgende Bemerkung unterstreicht, dass die Definition sinnvoll ist:

\begin{bem} Der transversale Schnitt zweier Untermannigfaltigkeiten einer glatten Mannigfaltigkeit ist selbst eine Untermannigfaltigkeit.
\end{bem}

Der transversale Schnitt wird benötigt, um die Art und Weise, wie die Diagonale herausgeschnittenen werden muss, zu spezifizieren:

\begin{satz}[Verminderte Giebelmannigfaltigkeiten]\label{vergieb}
Sei $W$ eine glatte, kompakte, berandete Mannigfaltigkeit und $(M,\partial M)\subset (W,\partial W)$ eine eingebettete Untermannigfaltigkeit mit Rand $\partial M=M\pitchfork \partial W$. Sei weiterhin $(U,\partial U)$ eine berandete Umgebung der Diagonalen $\bigtriangleup_{W}\subset W\times W$ mit transversalem Schnitt $\partial U\pitchfork (M\times M)$ und Innerem $\ring{U}=U\setminus\partial U$.\\
Dann ist der verminderte Giebel \[
\widetriangle{\left(M,\partial M\right)_{\ring{U}}}:=\left(\widetriangle{M}\setminus p_\tau^M\left(\ring{U}\right),\widetriangle{\partial M}\setminus p_\tau^M\left(\ring{U} \cap \left(\partial M\times \partial M\right)\right)\right)
\]
eine glatte, kompakte $2dim(M)-$dimensionale $\partial$-Mannigfaltigkeit.
\end{satz}

Sind $(W,\partial W)$ und $(M,\partial M)\subset (W,\partial W)$ sogar zusätzlich orientiert und besitzen gerade Dimension, so erbt $\widetriangle{\left(M,\partial M\right)_{\ring{U}}}$ die Orientierung wie im vorherigen Abschnitt beschrieben. Das und die Kompaktheit der verminderten Giebelmannigfaltigkeit liefern die Existenz einer wohldefinierten Fundamentalklasse \[
\left[ \widetriangle{\left(M,\partial M\right)_{\ring{U}}}\right] \in H^{sing}_{2dim(M)}\left(\widetriangle{\left(M,\partial M\right)_{\ring{U}}};\mathbb Z\right).
\]

In \v{C}ech-Homologie liefert das ein sehr brauchbares Ergebnis:

\begin{satz}\label{cefu}
Sei $(M,\partial M)\subset (W,\partial W)$ eine eingebettete, kompakte, gerade-dimensionale Untermannigfaltigkeit mit den Eigenschaften aus Satz \ref{vergieb}. Zudem existiere eine Orientierung. Dann besitzt der Giebel $\widetriangle{\left(M,\partial M\right)}$ eine nicht triviale, relative \v{C}ech-Homologieklasse\[
\left[\widetriangle{\left(M,\partial M\right)}\right] \in \check{H}_{2dim(M)}\left(\widetriangle{\left(M,\partial M\right)};\mathbb Z\right).
\]
\end{satz}

\begin{bew}[Idee]
Bemerke zunächst, dass die Menge der Umgebungen, die die Eigenschaften aus Satz \ref{vergieb} erfüllen, (stark) kofinal in der gerichteten Menge aller offenen Umgebungen der Diagonale liegen. Das heißt, dass zu einer beliebigen offenen Umgebung $V$ der Diagonalen eine Umgebung $U\subset V$ der Diagonalen gefunden werden kann, die die Eigenschaften aus Satz \ref{vergieb} erfüllt. Da jede Umgebung symmetrisiert werden kann, gehen wir hier immer von bereits symmetrischen Umgebungen aus.
Es ist\[
\check{H}_{2dim(M)}\left(\widetriangle{\left(M,\partial M\right)};\mathbb Z\right)\cong \varprojlim\left\{H_{2dim(M)}^{sing}\left(\widetriangle{M},V;\mathbb Z\right) \mid V\supset \widetriangle{\partial M} \text{ offen } \right\}.
\]
Per Definition enthalten die Umgebungen $V$ auch immer die Diagonale $\bigtriangleup_{M}$. Für jede dieser Umgebungen $V$ erhalten wir also $U\subset V$, welches zusätzlich die Eigenschaften aus Satz \ref{vergieb} erfüllt. Für dieses existiert eine (nicht triviale) Fundamentalklasse\[
\left[\widetriangle{\left(M,\partial M\right)}_{\ring{U}}\right]\in H_{2dim(M)}^{sing}\left(\widetriangle{(M,\partial M)}_{\ring{U}};\mathbb Z \right).
\]Für verschiedene Umgebungen $V_2\subset V_1$, gibt es dann auch Umgebungen $U_2\subset U_1$ und die Fundamentalklassen werden von den durch die Inklusionen induzierten Homomorphismen aufeinander abgebildet. Dadurch erhalten wir für offene Umgebungen $V$ von $\widetriangle{\partial M}$ nicht triviale Homologieklassen in $H_{2dim(M)}^{sing}(\widetriangle{M},V;\mathbb Z)$, welche ebenso die Eigenschaft haben, von den jeweiligen Inklusionen aufeinander abgebildet zu werden.\\ Diese sind wohldefiniert, da für zwei verschiedene Wahlen der Mengen $U\subset V$, die die Eigenschaften aus Satz \ref{vergieb} erfüllen, eine weitere Menge $\tilde{U}$ gefunden werden kann, die die entsprechenden Eigenschaften erfüllt, und die beide vorherigen Mengen enthält.\\ Es kann dann gezeigt werden, dass die von $\tilde{U}$ kommende Homologieklasse zu beiden vorher resultierenden Homologieklassen homolog ist, was die beiden zu homologen Homologieklassen macht.\\
Aus Satz \ref{invlim} ist die Gestalt des hier zu betrachtenden Limes bekannt, daher ist nun einzusehen, dass das Element, das aus den im vorherigen Absatz konstruierten Homologieklassen besteht, ein nicht triviales Element in \[\varprojlim\left\{H_{2dim(M)}^{sing}\left(\widetriangle{M},V;\mathbb Z\right) \mid V\supset \widetriangle{\partial M} \text{ offen } \right\}\cong \check{H}_{2dim(M)}\left(\widetriangle{\left(M,\partial M\right)};\mathbb Z\right)\text{ ist. }\]
\end{bew}

Diese Vorarbeit rechtfertigt, dass wir $\widetriangle{\left(M,\partial M\right)}$ als \textit{Giebelmannigfaltigkeit} mit der Klasse aus Satz \ref{cefu} als \textit{Fundamentalklasse} bezeichnen. Nun wird allmählich klar, was für eine wünschenswerte Eigenschaft die Dachabbildung hat:\\ Sie bildet unter den passenden Voraussetzungen die Fundamentalklasse einer Mannigfaltigkeit $M$ auf die Fundamentalklasse der zugehörigen Giebelmannigfaltigkeit ab.

\begin{satz}
Sei $W$ eine glatte, kompakte, berandete, orientierte Mannigfaltigkeit und $(M,\partial M)\subset (W,\partial W)$ eine eingebettete Untermannigfaltigkeit wie in Satz \ref{vergieb} gefordert. Diese habe weiterhin gerade Dimension k. Dann gilt \[ 
\widehat{\left[M,\partial M\right]}_{\mathbb Z}=\left[\widetriangle{(M,\partial M)}\right],
\]
wobei  $\left[M,\partial M\right]_{\mathbb Z}$ die Fundamentalklasse von $(M,\partial M)$ und $\left[\widetriangle{(M,\partial M)}\right]\in \check{H}_{2k}(\widetriangle{(M,\partial M)};\mathbb Z)$ die Fundamentalklasse der Giebelmannigfaltigkeit $\widetriangle{(M,\partial M)}$ bezeichnet.
\end{satz}

\begin{proof}
Da $M$ als glatte Mannigfaltigkeit trianguliert werden kann, kann auch zu jedem $k-$dimensionalen Simplex $s$ von $M$ ein affiner Isomorphismus $\sigma_s\colon \Delta_k\rightarrow s $ gewählt werden, so dass die Fundamentalklasse $\left[M,\partial M\right]_{\mathbb Z}$ von $(M,\partial M)$ durch den folgenden k-Zykel repräsentiert werden kann:\[
\sum_s{\sigma_s}\in C_k^{sing}(M,\partial M;\mathbb Z)\footnote{Vgl. \cite{1062.55001}, S. 138.},
\] wobei s die $k-$Simplizes von $M$ durchläuft. Man stelle sich die $k$-Simplizes durchnummeriert vor und es sei $n$ ihre Anzahl.\\
Es ist also \[
\widehat{\left[M,\partial M\right]}_{\mathbb Z}=\left[\widehat{\sum_s{\sigma_s}}\right]=\left[\sum\limits_{\genfrac{}{}{0pt}{1}{s<t}{1 \leq s,t \leq n}}{(p_{\tau}^X)_{\sharp}(\sigma_s\times\sigma_t)} \right]\in H^{sing}_{2k}\left(\widetriangle{(M,\partial M)};\mathbb Z\right).
\]Da $k$ gerade ist, gilt $(p_{\tau}^X)_{\sharp}(\sigma_s\times\sigma_t)=(p_{\tau}^X)_{\sharp}(\sigma_t\times\sigma_s)$ für alle $k-$Simplizes $s,t$ von $M$. Aus diesem Grund ist aber $\sum\limits_{\genfrac{}{}{0pt}{1}{s<t}{1 \leq s,t \leq n}}{(p_{\tau}^X)_{\sharp}(\sigma_s\times\sigma_t)}$ die Summe über alle $2k$-Simplizes von $\widetriangle{(M,\partial M)} $ und repräsentiert somit die Fundamentalklasse der Giebelmannigfaltigkeit $\widetriangle{(M,\partial M)} $.\\ Allerdings kann das so nur gelten, wenn die Triangulierungen von $(M,\partial M)$ und $\widetriangle{(M,\partial M)}$, welches im Allgemeinen gar keine Mannigfaltigkeit ist, zusammenpassen. Insbesondere sollte es an den Schnittstellen, an denen die Umgebung der Diagonalen in $M\times M$ herausgeschnitten wurde, eine Triangulierung geben, die mit der von $M$ zusammen passt.\\ Da nun aber die verminderte Giebelmannigfaltigkeit tatsächlich eine glatte Mannigfaltigkeit ist, kann sie in ausreichender Entfernung von der herausgeschnittenen Diagonalen sogar auf die selbe Art wie $M$ trianguliert werden. Da Triangulierungen ausdehnbar sind, folgt das Gewünschte.
\end{proof}

%% file: ausblick.tex
Eine Homologietheorie, die sehr geometrisch definiert ist, ist die Bordismustheorie. Hier werden für einen topologischen Raum $X$ Abbildungen $f\colon M\rightarrow X$ betrachtet, wobei $M$ eine geschlossene, glatte, n-dimensionale Mannigfaltigkeit ist.\\ 
Solche sogenannten singulären $n$-Mannigfaltigkeiten $(M_1,f_1)$ und $(M_2,f_2)$ heißen bordant, wenn es eine kompakte $(n+1)$-dimensionale Mannigfaltigkeit W, eine Abbildung $F\colon W\rightarrow X$, eine disjunkte Zerlegung des Randes $\partial W=\partial_0 W\amalg \partial_1 W$ und Diffeomorphismen $u_i\colon M_i\rightarrow \partial_i W$ für $i=0,1$ gibt, so dass $F\circ u_i=f_i$ für $i=0,1$ gilt. Sie heißen also bordant, wenn sie gemeinsam den Rand einer höherdimensionalen Mannigfaltigkeit ergeben. "`Bordant"' ist eine Äquivalenzrelation, und es ist möglich, auf den Äquivalenzklassen $N_n(X)$ der Bordismusklassen $(M,f)$ durch die disjunkte Summe die Struktur einer abelschen Gruppe zu definieren.\\  Diese Konstruktion lässt sich auf Raumpaare $(X,A)$ und Abbildungen $(M,\partial M)\rightarrow (X,A)$ erweitern. Insgesamt resultiert eine Homologietheorie, die allerdings das Dimensionsaxiom nicht erfüllt und  die unorientierter Bordismus genannt wird.\\
Es drängt sich die Frage auf, ob eine Dachabbildung auch im Bordismus existieren würde.\\ Da bei der Frage, inwiefern $\widetriangle{M}$ für eine $\partial$-Mannigfaltigkeit $M$ wieder eine Mannigfaltigkeitsstruktur aufweist, vor allem die Diagonale in $M\times M$ eine Schwierigkeit darstellt, wird es auch hier nötig sein, sich besonders um die Diagonale zu kümmern. Außerhalb der Diagonalen entsteht durch die Giebelbildung im Grunde nichts anderes als vorher, denn dort ist $M\times M$ eine zweifache Überlagerung von $\widetriangle{M}$ (ohne das Bild der Diagonalen). Die Diagonale ist aber auf den ersten Blick nicht so einfach in das dadurch entstehende Bild einzufügen.